\documentclass[12pt, reqno]{amsart}
\usepackage{amssymb}
\usepackage{amsthm}
\usepackage{mathtools}
\usepackage[normalem]{ulem}
\usepackage{amsmath, geometry}
\usepackage{listings}
\usepackage[utf8]{inputenc}
\usepackage[english]{babel}
\usepackage[dvipsnames]{xcolor}
\usepackage{amsfonts}
\usepackage{lipsum}

\usepackage{tikz}
\usetikzlibrary{cd}
\usepackage{mathrsfs}
\usetikzlibrary{arrows}
\usepackage{scrextend}
\usepackage[makeroom]{cancel}
\usepackage{enumitem}
\usepackage{float}
\usepackage{graphicx} 

\usepackage[backend=biber,
    style=numeric,
    maxbibnames=99,
    giveninits=true,
    url=false,        
    doi=false,        
    isbn=false,      
    eprint=true
    ]{biblatex}
\definecolor{qqffff}{rgb}{0,1,1}
\definecolor{ttffcc}{rgb}{0.2,1,0.8}
\definecolor{ududff}{rgb}{0.30196078431372547,0.30196078431372547,1}
\definecolor{qqqqff}{rgb}{0,0,1}
\definecolor{ffffff}{rgb}{1,1,1}
\definecolor{ffqqqq}{rgb}{1,0,0}
\definecolor{qqffqq}{rgb}{0,1,0}
\definecolor{cqcqcq}{rgb}{0.7529411764705882,0.7529411764705882,0.7529411764705882}
\definecolor{xfqqff}{rgb}{0.4980392156862745,0,1}
\definecolor{ccwwff}{rgb}{0.8,0.4,1}
\definecolor{zzttqq}{rgb}{0.6,0.2,0}
\definecolor{cczzff}{rgb}{0.8,0.4,1}

\theoremstyle{definition}
\newtheorem*{assumptions}{Assumptions}

\theoremstyle{plain}
\newtheorem{theorem}{Theorem}[subsection]
\newtheorem{lemma}[theorem]{Lemma}
\newtheorem{corollary}[theorem]{Corollary}
\newtheorem{proposition}[theorem]{Proposition}

\theoremstyle{definition}
\newtheorem{example}[theorem]{Example}
\newtheorem{fact}[theorem]{Fact}
\newtheorem{notation}[theorem]{Notation}
\newtheorem{definition}[theorem]{Definition}

\newtheorem{remarkn}[theorem]{Remark}

\newcommand{\hei}{\mbox{height}}
\newcommand{\dist}{\text{dist}}

\newcommand\tcr[1]{\textcolor{red}{#1}}
\newcommand\tcb[1]{\textcolor{blue}{#1}}

\newcommand{\tp}{\mbox{type}}
\newcommand{\depth}{\mbox{depth}}
\newcommand{\hh}{\mbox{height}}

\newcommand{\Ext}{\mbox{Ext}}
\newcommand{\R}{\mathcal{R}}
\newcommand{\B}{\mathcal{B}}

\newcommand{\kk}{\Bbbk}

\newcommand{\mrad}{\mbox{m-rad}}

\newcommand\gen[1]{\left\langle #1 \right\rangle}

\newcommand\Set[2]{\left\lbrace #1 \mid #2 \right\rbrace}
\newcommand\set[1]{\left\lbrace #1 \right\rbrace}

\newcommand\tco[1]{\textcolor{orange}{#1}}

\newcommand{\NN}{\mathcal{N}}
\newcommand{\II}{\mathcal{I}}

\newcommand{\calS}{\mathcal{S}}
\newcommand{\calF}{\mathcal{F}}

\allowdisplaybreaks

\begin{document}

\title[Open Neighborhood Ideals of Well Totally Dominated Trees]{Open Neighborhood Ideals of Well-Totally Dominated Trees are Cohen-Macaulay}

\author[Lim, et al]{Jounglag Lim}
\address{School of mathematical and statistical sciences, Clemson University, Clemson, SC 29634}
\email{joungll@clemson.edu}

\author[]{James Gossell}
\address{Department of mathematics and statistics, University of Alaska, Fairbanks, AK 99775}
\email{jegossell@alaska.edu}

\author[]{Keri Ann Sather-Wagstaff}
\address{Framingham State University, Framingham, MA 01701}
\email{ksatherwagstaff@framingham.edu}

\begin{abstract}
We introduce and investigate the open neighborhood ideal $\NN(G)$ of a finite simple graph $G$. We describe the minimal primary decomposition of $\NN(G)$ in terms of the minimal total dominating sets (TDSs) of $G$. Then we prove that the open neighborhood ideal of a tree is Cohen-Macaulay if and only if the tree is well-totally dominated (WTD) and calculate the Cohen-Macaulay type.
\end{abstract}

\maketitle
\setcounter{tocdepth}{1}
\tableofcontents

\section{Introduction}

Commutative algebra and combinatorics have a rich history of intersections.
In this paper, we focus on connections between commutative algebra and graph theory.
This approach begins with work of Villarreal~\cite{MR1031197} who introduced and investigated the \emph{edge ideal} $\II(G)$ of a finite, simple graph $G$, which is the ideal generated by the edges of $G$ in an appropriate polynomial ring. He described the minimal primary decomposition of $\II(G)$ in terms of the minimal vertex covers of $G$ and proved that the edge ideal of a tree is Cohen-Macaulay if and only if the tree is unmixed (well covered). He also gave explicit descriptive and constructive characterizations of all unmixed trees.
Much subsequent research has expanded our understanding of the relations between the algebraic properties of $\II(G)$ and the graph-theoretic properties of $G$; see, e.g., \cite{SCM_graphs,wei2019cohenmacaulay,MR2302553,MR2231097,MR1425502}.

Variations on this algebraic construction relate to different properties of the graph.
For instance, Conca and De Negri \cite{MR1666661} introduced the $r$-path ideal $\II_r(G)$ which is generated by the paths of length $r$ in $G$ and is connected to $r$-path covers;
the edge ideal is the special case $r = 1$, i.e., $\II_{1}(G) = \II(G)$.
Campos, et al.~\cite{MR3175038} characterized the Cohen-Macaulay $r$-path ideals of trees.
Sharifan and Moradi \cite{MR4132629} introduced the closed neighborhood ideal of $G$ which is generated by the closed neighborhoods in $G$ and is connected to domination sets.
Honeycutt and Sather-Wagstaff~\cite{MR4445927} characterized the trees whose closed neighborhood ideals are Cohen-Macaulay as those that are unmixed (well-dominated tree).

This paper investigates a similar construction based on open neighborhoods: the \emph{open neighborhood ideal} $\NN(G)$ of $G$ is the ideal generated by the open neighborhoods of $G$. 
We define this formally in Definition~\ref{def. open neig. ideal} below. 
Section~\ref{subsec241204a} contains some foundational information. 
It includes Theorem~\ref{decomp of ONI} which describes the minimal primary decomposition of $\NN(G)$ in terms of the minimal total dominating sets (TDSs) of $G$.

Section~\ref{sec241204b} consists of graph-theoretic results, mainly, the descriptive characterization of well-totally dominated (WTD) trees; see Theorem~\ref{thm. char. of WTD trees}.
The proofs of the results in this section can be found in \cite{COURAGE_GT}.
Then in Section~\ref{sec241204c}, we show that the open neighborhood ideals of WTD trees are Cohen-Macaulay; see Corollary~\ref{cor. CM of ONI and SON} which follows from the fact that the corresponding Stanley-Reisner simplicial complex is shellable, proved in Theorem~\ref{thm. stable complex is shellable}. 
The current paper ends with Section~\ref{sec241204d} where we use our results to compute the Cohen-Macaulay type of open neighborhood ideal of WTD trees in Theorem~\ref{thm. CM type of WTD trees}.

\begin{assumptions}
Throughout, we assume that $G$ is a finite, simple graph, $\kk$ is a field, and $A$ is a nonzero commutative ring with unity.
The vertex set and the edge set of $G$ are $V(G)$ and $E(G)$, respectively.
If the graph we consider is clear in context, then we write $V := V(G)$ and $E := E(G)$.
Since $G$ is simple, we denote each edge $e \in E$ as $uv=vu$ where $u,v \in V$ are the distinct vertices incident to $e$.
Let $A[V(G)] = A[V]$ be the polynomial ring over $A$ whose variables are the vertices of $G$.
We set $R := \kk[V]$ or $R := A[V]$ depending on the~context.
\end{assumptions}

\noindent\textbf{Acknowledgement.} We are gratful to Clemson University's School of Mathematical and Statistical Sciences for financial support for the 2020 summer REU COURAGE (Clemson Online Research on Algebra and Graphs Expanded) where this research started.

\section{Total Dominating Sets and Open Neighborhood Ideals}
\label{subsec241204a}

This section contains some graph-theoretical, foundational information and Theorem~\ref{decomp of ONI} which describes the minimal primary decomposition of $\NN(G)$ in terms of the minimal total dominating sets (TDSs) of $G$.

\subsection{Foundations}\label{subsec. foundations}

\begin{definition}\label{def. open neighborhood}
    For $v \in V(G)$, the \emph{open neighborhood of $v$ in $G$} is the set of vertices 
$$
N_G(v) := \Set{u \in V}{uv \in E}.
$$
For any subset $S \subseteq V(G)$, we define 
$$
N_G(S) := \bigcup_{v \in S} N_G(v).
$$
If the graph $G$ is clear in the context, we set $N(v):=N_G(v)$ and $N(S):=N_G(S)$.
    For $U \subseteq V$, define the \emph{monomial of $U$} in $A[V]$ by  
    $X_U := \prod_{v \in U} v$.
\end{definition}

%

\begin{definition}\label{def. S-TDS, SONI}
Fix $S \subseteq V$.
A set $D \subseteq V$ is an \emph{$S$-totally dominating set} (\emph{$S$-TDS}) of $G$ if $N(D) \supseteq S$. An $S$-TDS $D$ of $G$
is \emph{minimal} if there is no proper subset $D'$ of $D$ such that $N_G(D') \supseteq S$, i.e., if $D$ does not properly contain another $S$-TDS of $G$.
We say that $G$ is $S$-\emph{well-totally dominated} ($S$-WTD) if every minimal $S$-TDS of $G$ has the same size.
The $S$-open neighborhood ideal of $G$ is the following ideal in $R = A[V]$
    $$
    \NN_S(G) := \gen{X_{N(v)}\mid v \in S}.
    $$
\end{definition}

\begin{example}\label{ex. P5 STDset}
Consider the path $P_5$ with 5 edges and set $R = A[V]$ and $S = \set{v_2,v_4,v_6}$.
\begin{figure}[ht]
\centering
\includegraphics{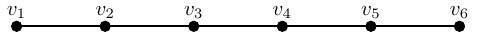}
\label{P5}
\end{figure}
Then the $S$-open neighborhood ideal of $P_5$ is 
    \begin{align*}
        \NN_S(P_5) = \gen{X_{N(v_2)}, X_{N(v_4)}, X_{N(v_6)}}
        = \gen{v_1v_3,v_3v_5,v_5}
        = \gen{v_1v_3,v_5}
    \end{align*}
which decomposes irredundantly as
    \begin{align*}
        \NN_S(P_5) &= \gen{v_1v_3,v_5}
        = \gen{v_1,v_5} \cap \gen{v_3,v_5}.
    \end{align*}
    It is straightforward to see that the sets $\set{v_1,v_5}$ and $\set{v_3,v_5}$ corresponding to the ideals in this decomposition are exactly the minimal $S$-TDSs of $P_5$.
    In particular, $P_5$ is $S$-WTD.
\end{example}

The following fact is used in the proof of Theorem~\ref{thm. decomp of SON} which gives a general decomposition of $\NN_S(G)$.

\begin{fact}\label{subset relation of S-TDS}
   Let $S \subseteq V$. The following statements are straightforward to verify.
    \begin{itemize}
        \item[(1)] For any $V' \subseteq V'' \subseteq V$, if $V'$ is an $S$-TDS of $G$, then $V''$ is an $S$-TDS of $G$.
        \item[(2)] Every $S$-TDS of $G$ contains a minimal $S$-TDS of $G$. 
    \end{itemize}
\end{fact}

\subsection{Decompositions}

\begin{lemma}\label{lem. S-TDS iff P_V' contains SON}
    Let $S,V' \subseteq V$ and $R = A[V]$.
    Then $V'$ is an $S$-TDS of $G$ if and only if $\NN_S(G) \subseteq \gen{V'}$.
\end{lemma}

\begin{proof}
    For the forward implication, assume that $V'$ is an $S$-TDS.
    Consider a generator $X_{N(s)} \in \NN_S(G)$ with $s \in S$.
    Since $V'$ is an $S$-TDS, there exists some vertex $v \in V'$ such that $v \in N(s)$.
    Hence we get $v | X_{N(s)}$, so $X_{N(s)} \in (v)R \subseteq \gen{V'}$.
    Thus $\NN_S(G) \subseteq \gen{V'}$.

    For the converse implication, assume that $\NN_S(G) \subseteq \gen{V'}$.
    Let $u \in S$.
    Then $X_{N(u)} \in \NN_S(G) \subseteq \gen{V'}$.
    Hence there exists some $v \in V'$ such that $v | X_{N(u)}$.
    By definition of $X_{N(u)}$, we must have $v \in N(u)$.
    Thus $u \in N(v) \subseteq N(V')$.
    Since $u$ is an arbitrary vertex in $S$, we get $S \subseteq N(V')$, so $V'$ is an $S$-TDS.
\end{proof}

The following theorem describes the decomposition of $S$-open neighborhood ideals using $S$-totally dominating sets. 

\begin{theorem}\label{thm. decomp of SON}
    Let $S \subseteq V$, and $R = A[V]$.
    The $S$-open neighborhood ideal has the following decomposition
    $$
    \NN_S(G) = \bigcap_{D} \gen{D} = \bigcap_{D\ min} \gen{D}
    $$
    where the first intersection is taken over all $S$-TDSs of $G$, and the second intersection is taken over all minimal $S$-TDSs of G.
    Moreover, the second decomposition is irredundant.
\end{theorem}

\begin{proof}
    For any $A,B \subseteq V$, we have $\gen A \subseteq \gen B$ if and only if $A \subseteq B$.
    Hence the second intersection is irredundant.
    Now let $V' \subseteq V$ be an $S$-TDS which is not minimal.
    Then $V'$ contains a minimal $S$-TDS by Fact~\ref{subset relation of S-TDS}.
    So, we have $\bigcap_{D}\gen{D} = \bigcap_{D \neq V'}\gen D$.
    Since $V$ is finite, by repeating the same argument finitely many times, we can conclude that $\bigcap_{D} \gen{D} = \bigcap_{D\ min} \gen{D}$.

    By Lemma~\ref{lem. S-TDS iff P_V' contains SON}, we get $\NN_S(G) \subseteq \bigcap_{D} \gen{D}$.
    For the containment $\NN_S(G)\supseteq \bigcap_{D} \gen{D}$, since $\NN_S(G)$ is a square-free monomial, there are subsets $D_1,\dots,D_k\subseteq V$ such that $\NN_{S}(G) = \bigcap_{i = 1}^k \gen{D_i}$.
    For any index $j \in [k]$, we have $\NN_S(G) \subseteq \gen{D_j}$, which implies that $D_j$ is an $S$-TDS by Lemma~\ref{lem. S-TDS iff P_V' contains SON}.
    Thus we get $\NN_S(G) = \bigcap_{i = 1}^k \gen{D_i} \supseteq \bigcap_{D} \gen{D}$.
\end{proof}

Now we consider a specific case when $S = V$, the entire vertex set.

\begin{definition}\label{def. TDS}
A subset $D \subseteq V(G)$ is a \emph{total dominating set (TDS) of $G$} if $N(D) = V(G)$.
We say $D$ is a \emph{minimal TDS} if no proper subset of $D$ is a TDS.
We say $G$ is \emph{WTD} if every minimal TDS of $G$ has the same size.
\end{definition}

The minimal TDS problem on graphs is a well studied subject in graph theory \cite{MR3060714}, and finding the size of the smallest minimal TDS of a graph turns out to be an NP-complete problem \cite{MR752045}.
Also, finding and characterising WTD graphs is an interesting problem \cite{MR942485, MR4197372}. 
In this paper, we focus on the WTD-ness of trees to establish our goal.

\begin{definition}\label{def. open neig. ideal}
    The \emph{open neighborhood ideal of $G$} in $R = A[V]$ is the ideal generated by the monomials of the open neighborhoods of the vertices of $G$
    $$
    \NN(G) := \NN_V(G) = \gen{X_{N_G(v)} \mid v \in V}.
    $$
\end{definition}

\begin{example}\label{ex. ONI of P_6 plus a leaf}
    Let $T$ be the following tree with vertex set $V = \set{v_1,v_2,\dots,v_8}$.
    \begin{figure}[ht]
        \centering
        \includegraphics{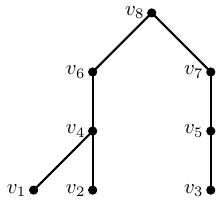}
        \label{fig:enter-label}
    \end{figure}
    
    \noindent Then $R = A[V] = A[v_1,\dots,v_8]$, and we have
    \begin{align*}
        \NN(T) 
        = \gen{v_4,v_5,v_1v_2v_6,v_3v_7,v_4v_8, v_5v_8, v_6v_5}
        = \gen{v_4,v_5,v_1v_2v_6,v_3v_7, v_6v_5}.
    \end{align*}
\end{example}

We get the following decomposition theorem of open-neighborhood ideals as a corollary of Theorem~\ref{thm. decomp of SON}. When $A$ is a field, this gives the (minimal) primary decomposition of $\NN(G)$.

\begin{theorem}\label{decomp of ONI}
The open neighborhood ideal of $G$ in $R = A[V]$ has the following decomposition
    $$
    \NN(G) = \bigcap_{V'}\gen{V'} = \bigcap_{V'\ min}\gen{V'}
    $$
    where the first intersection is taken over all TDSs of $G$, and the second intersection is taken over all minimal TDSs of $G$.
    Moreover, the second decomposition is irredundant.
\end{theorem}

\begin{example}
\label{ex. Path with 6 vert.}
Consider the tree $T$ and its open neighborhood ideal from Example~\ref{ex. ONI of P_6 plus a leaf}.
One can show that the following sets are the minimal TDSs of~$T$:
\begin{align*}
    \set{v_4,v_5,v_6,v_7} & & \set{v_1,v_4,v_5,v_7} & & \set{v_2,v_4,v_5,v_7} & & \set{v_4,v_5,v_6,v_3}.  
\end{align*}
Thus $T$ is WTD, and by Theorem~\ref{decomp of ONI}, we get
\begin{align*}
    \NN(T) &= \gen{v_4,v_5,v_1v_2v_6,v_3v_7, v_6v_5}\\
    &= \gen{v_4,v_5,v_6,v_7} \cap \gen{v_1,v_4,v_5,v_7} \cap \gen{v_2,v_4,v_5,v_7} \cap \gen{v_4,v_5,v_6,v_3}.
\end{align*}
On the other hand it is straightforward to show that the path $P_4$
\begin{center}
    \begin{tikzpicture}[every edge/.style={draw,thick}]
    \node (l1) at (0,0) {$\ell_1$};
    \node (s1) at (2,0) {$s_1$};
    \node (r)  at (4,0) {$u$};
    \node (s2) at (6,0) {$s_2$};
    \node (l2) at (8,0) {$\ell_2$};
    \draw (l1) -- (s1) -- (r) -- (s2) -- (l2);
\end{tikzpicture}
\end{center}
is not WTD because
\[    
\NN(P_4) = \gen{s_1, \ell_1u, s_1s_2, u\ell_2, s_2} = \gen{s_1,s_2,\ell_1,\ell_2} \cap \gen{s_1,s_2,u}.
\]
\end{example}

\section{Characterizations of WTD trees}
\label{sec241204b}

Our goal in this section is to overview a characterization of WTD trees.
It is a descriptive characterization, allowing one to detect whether a given tree $T$ is WTD; see Theorem~\ref{thm. char. of WTD trees}.
Essentially, $T$ has two noteworthy ``interior" subgraphs, $T_\R$ and $T_\B$, which are induced subforests that arise from a red-blue coloring of $T$, and such that $T$ is WTD if and only if $T_\R$ and $T_\B$ are WTD \cite[Corollary 5.8]{COURAGE_GT}.
The connected components of $T_\R$ and $T_\B$ have a property that we call ``balanced."
The main result of which is the descriptive characterization of balanced trees in Theorem~\ref{thm. char. of WTD delt. trees}.

\begin{assumptions}
In this section, we assume that $T = (V,E)$ is a tree with a 2-coloring $\chi: V \to \set{\R,\B}$ where $\R$ and $\B$ stands for red and blue color, respectively.
\end{assumptions}

\subsection{Balanced trees, RD-, BD-sets}

We begin with the definition of height.

\begin{definition}
Let $G$ be a graph with at least one leaf (vertex of degree 0).
The \emph{height} of a vertex $v \in V$ is given by
$$
\hh(v) := \min\Set{\dist(v,\ell)}{\ell \mbox{ is a leaf in } G}.
$$ 
If $v$ is an isolated vertex, then we set $\hh(v) = 0$ as a convention. 
We denote 
$$
V_k = V_k(G) := \{v \in V\ |\  \hh(v) = k\}
$$ 
and the \emph{height} of $G$ is the integer 
$$
\hh(G) := \max\Set{k \in \mathbb{N}}{V_k \neq \emptyset}.
$$
For instance, every leaf of a graph has height 0, and every non-leaf support vertex (vertex that are adjacent to a leaf) has height~1.
\end{definition}

\begin{definition}\label{def. delt. tree}
    The tree $T$ is \emph{balanced} if no two vertices of the same height are adjacent.
\end{definition}

Proposition~\ref{prop. equiv. delt. tree} gives an equivalent definition of balanced trees.

\begin{proposition}[\protect{\cite[Proposition 3.4]{COURAGE_GT}}]\label{prop. equiv. delt. tree}
The following conditions on a tree $T$ are equivalent:
\begin{enumerate}
	\item $T$ is balanced.
	\item Any two vertices of the same height have the same color under $\chi$.
	\item Every leaf has the same color.
\end{enumerate}
\end{proposition}

\begin{example}
    Consider the following trees with 2-coloring.
    \begin{figure}[H]
        \centering
        \includegraphics{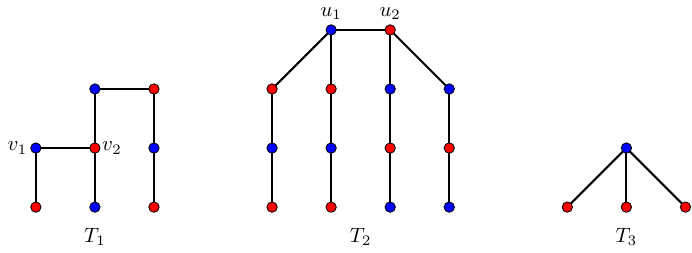}
        \caption{Trees with 2-colorings}
        \label{fig. 3 trees with coloring}
    \end{figure}        
    \noindent The tree $T_1$ is not balanced since $\hh(v_1) = 1 = \hh(v_2)$ but $v_1v_2 \in E(T_1)$.
    Similarly, $u_1$ and $u_2$ show that $T_2$ is not balanced.
    We can also deduce this from Proposition~\ref{prop. equiv. delt. tree} since the leaves in each tree do not have the same color.
    On the other hand, $T_3$ is a balanced tree by definition, or since all of its leaves share the same color.
\end{example}

The following corollary is a quick consequence of Proposition~\ref{prop. equiv. delt. tree}.

\begin{corollary}\label{cor. same height parity -> same color}
    Let $T$ be a balanced tree.
    Then all vertices of even height have the same color, and all vertices of odd height have the same color.
\end{corollary}

Next, we define two types of $S$-TDSs for balanced trees.
Fact~\ref{fact. TDS iff RD-set U BD-set} is used in Section~\ref{sec241204c}.

\begin{notation}
    Denote the set of red and blue vertices of $T$ by $V_{\R}$ and $V_{\B}$, respectively; i.e., for all $C \in \set{\R,\B}$, define $V_C := \chi^{-1}(C)$.
\end{notation}
    
\begin{definition}
    Let $D \subseteq V$.
    Then $D$ is a \emph{red-dominating set} (RD-set) if $N(D) = V_\R$. 
    And $D$ is a \emph{blue-dominating set} (BD-set) if $N(D) = V_\B$.
\end{definition}

Note that by the definition of 2-coloring, if $D$ is an RD-set, then $D \subseteq V_\B$, and $D \subseteq V_\R$ if $D$ is a BD-set.
Hence we can state the following as a fact.

\begin{fact}\label{fact. TDS iff RD-set U BD-set}
    Let $D' \subseteq V_\B$, $D'' \subseteq V_\R$, and $D = D' \cup D''$. 
    Then
    \begin{itemize}
        \item[(1)] $D$ is a TDS if and only if $D'$ and $D''$ are red-dominating and blue-dominating, respectively.
        \item[(2)] $D$ is a minimal TDS if and only if $D'$ and $D''$ are minimal red-dominating and minimal blue-dominating, respectively.
    \end{itemize}
\end{fact}

\begin{example}
Consider the tree $Y$ with vertices colored red and blue in Figure~\ref{fig. Yifan Graph colored}.
\begin{figure}[ht]
\centering
\includegraphics{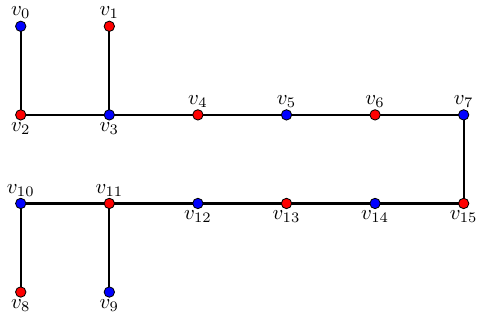}
\caption{Graph $Y$ with a 2-coloring}
\label{fig. Yifan Graph colored}
\end{figure}
The set $D := \{\tcr{v_2}, \tcb{v_3}, \tcb{v_5}, \tcr{v_6}, \tcb{v_{10}}, \tcr{v_{11}}, \tcb{v_{14}}, \tcr{v_{15}}\}$ is a minimal TDS of $Y$.
Now, we partition the set $D$ by the coloring: set  $D' := \{ \tcb{v_3}, \tcb{v_5},  \tcb{v_{10}},  \tcb{v_{14}}\}$ and $D'' := \{\tcr{v_2},\tcr{v_6},\tcr{v_{11}},\tcr{v_{15}}\}$.
Then it is straightforward to show that $D'$ and $D''$ are minimal red-dominating and blue dominating, respectively.
\end{example}

\subsection{Descriptive characterization of WTD balanced trees}\label{subsec. des. char. of WTD bal. trees}

Theorem~\ref{thm. char. of WTD delt. trees} is the main result of this subsection.
It allows us to detect when a balanced tree is WTD.
The case of an arbitrary tree is handled by Theorem~\ref{thm. char. of WTD trees} below.

\begin{theorem}[\protect{\cite[Theorem 4.12]{COURAGE_GT}}]
\label{thm. char. of WTD delt. trees}
Let $T$ be a balanced tree.
Then $T$ is WTD if and only if
\begin{enumerate}
    \item $\hei(T) \leq 3$,
    \item for all $v \in V_2$, $|N(v) \cap V_1| = 1$, and
    \item for all $v \in V_1$, $|N(v) \cap V_2| \leq 1$.
\end{enumerate}
\end{theorem}

As a remark, the inequality in (3) becomes an equality if $\hei(T) = 3$.
Now we define the interior graphs of a tree $T$ which is used to obtain our descriptive characterization of WTD trees in Theorem~\ref{thm. char. of WTD trees}.
The interior graphs are also used to compute the Cohen-Macaulay type of open neighborhood ideals in Section~\ref{sec241204d}.

\begin{definition}\label{def. interior graphs}
    We define the \emph{blue interior graph} of $T$ to be the subgraph $T_\B$ of $T$ induced by the set 
    $$
    V \setminus N[V_1 \cap \chi^{-1}(\B)]
    $$
    where $N[U] := N(U) \cup U$ is the \emph{closed neighborhood of $U$} for any $U \subseteq V$; i.e., $T_B$ is the subgraph of $T$ induced by deleting all blue support vertices together with their neighbors.
    Similarly, we define $T_\R$ to be the \emph{red interior graph} of $T$ induced by the set
    $$
    V \setminus N[V_1 \cap \chi^{-1}(\R)].
    $$
\end{definition}

\begin{example}\label{ex. decomp. of a tree}
Figure \ref{fig. 2-colored T} shows a 2-colored tree $T$, and Figure~\ref{fig. T_B T_R of T} below shows its blue and red interior graphs. This example shows that the interior graphs can be forests.
\begin{figure}[ht]
        \centering
        \includegraphics{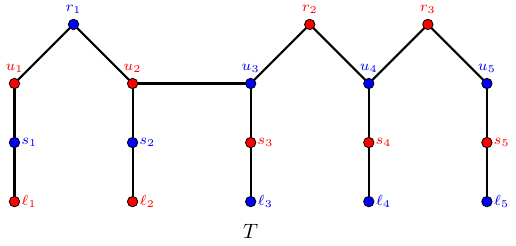}
        \caption{A 2-coloring of a tree $T$}
        \label{fig. 2-colored T}
\end{figure}
\begin{figure}[ht]
        \centering
        \includegraphics{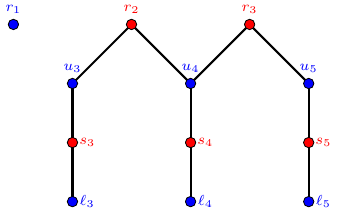}
        \qquad
        \vline
        \qquad
        \includegraphics{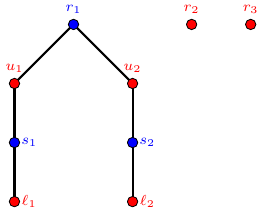}
        \caption{Interior graphs $T_\B$ (left) and $T_\R$ (right) of $T$ from Figure \ref{fig. 2-colored T}}
        \label{fig. T_B T_R of T}
\end{figure}    
\end{example}

\begin{example}\label{ex. T, T_B = T_R}
    Consider the tree $T$ in Figure \ref{fig. tree_with_all_h}.
    Its interior graphs $T_\B$ and $T_\R$ are shown in Figure  \ref{fig. T_B T_R of tree with all h}. 
    For this specific tree $T$, its interior graphs are isomorphic.

    \begin{figure}[ht]
        \centering
        \includegraphics{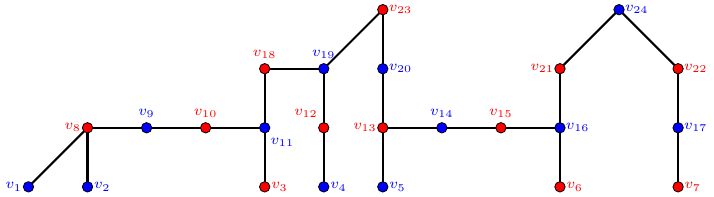}
        \caption{A tree $T$}
        \label{fig. tree_with_all_h}
    \end{figure}

    \begin{figure}[ht]
        \centering
        \includegraphics[width = 0.449\textwidth]{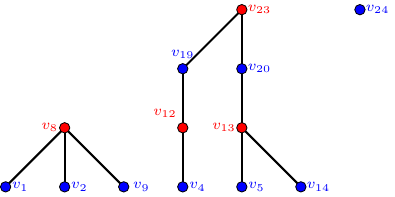}
        \quad
        \vline
        \quad
        \includegraphics[width = 0.449\textwidth]{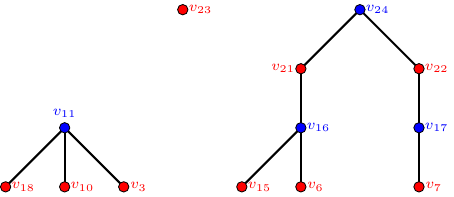}
        \caption{$T_\B$ (left) and $T_\R$ (right) of $T$ from Figure \ref{fig. tree_with_all_h}}
        \label{fig. T_B T_R of tree with all h}
    \end{figure}
\end{example}

Each connected component of the interior graphs in Examples  \ref{ex. T, T_B = T_R} and \ref{ex. decomp. of a tree} is a balanced tree.
Here is the characterization of WTD trees.

\begin{theorem}[\protect{\cite[Theorem 5.9]{COURAGE_GT}}]\label{thm. char. of WTD trees}
    A tree $T$ is WTD if and only if every connected component $T'$ of $T_\B$ and $T_\R$ has
    \begin{enumerate}
        \item $\hh(T') \leq 3$,
        \item for all $v \in V_2(T')$, we have $|N_{T'}(v) \cap V_1(T')| = 1$, and
        \item for all $v \in V_1(T')$, we have $|N_{T'}(v) \cap V_2(T')| \leq 1$.
    \end{enumerate}
    In other words, $T$ is WTD if and only if $T_\B$ and $T_\R$ are WTD.
\end{theorem}

\section{Shellability of Stable complexes}
\label{sec241204c}

Our goal in this section is to prove that the Stanley-Reisner complex \cite{Stanley_CCA} of open neighborhood ideals of WTD trees, called stable complexes, is shellable; see Theorem~\ref{thm. stable complex is shellable}.
In particular, we explicitly construct the shelling order of these complexes.
In subsection~\ref{subsec. stable complexes}, we define stable complexes of graphs and even-stable complexes of balanced trees, which correspond to the Stanley-Reisner complexes of open neighborhood ideals of graphs and odd-open neighborhood ideals of balanced trees, respectively.
Then we show that stable complexes of trees are the join of even-stable complexes of its interior graphs in Theorem~\ref{thm. stable comp. is join of even complexes}.
In subsection~\ref{subsec. Shellability of even stable complexes}, we show that the even-stable complexes of WTD balanced trees are shellable in Theorem~\ref{thm. even stable complex is shellable}, which implies the main theorem of this section.

\subsection{Stable complexes of graphs}\label{subsec. stable complexes}
In this subsection, we define the stable complexes of graphs and the even-stable complexes of balanced trees; see Definitions~\ref{def. stable complex} and \ref{def. even stable complex of delta trees}. 
We then show that the stable complexes of trees are the join of the even-stable complexes of its interior graphs; see Theorem~\ref{thm. stable comp. is join of even complexes}.
To start, we review some basic notions of simplicial complexes.

Let $V$ be a finite set. 
We say $\Delta \subseteq 2^V$ (where $2^V$ denotes the power set of $V$) is a \emph{simplicial complex} if $F \in \Delta$ and $G \subset F$ implies $G \in \Delta$.
Elements in $\Delta$ are called \emph{faces}, and the maximal elements (under inclusion) in $\Delta$ are called \emph{facets}.
We denote the set of all facets of $\Delta$ by $\mathcal{F}(\Delta)$.
If $\calF(\Delta) = \set{ F_1,\dots,F_m}$, then we also write $\Delta = \gen{F_1,\dots,F_m}$.
In this case, we say $\Delta$ is \emph{generated} by $F_1,\dots,F_m$.
The \emph{dimension} of a face $F \in \Delta$ is $\dim(F) := |F| - 1$, and the \emph{dimension} of $\Delta$ is $\dim(\Delta) := \max\Set{\dim(F)}{F \in \Delta}$.
We say $\Delta$ is \emph{pure} if every facet in $\Delta$ has the same dimension.

\begin{definition}\label{def. stable complex}
    Let $G = (V,E)$ be a graph.
    A set $V' \subseteq V$ is \emph{stable} if $V\setminus V'$ is a TDS.
    We say $V'$ is \emph{maximal} if $V\setminus V'$ is minimal.
    The \emph{stable complex} of $G$ is the simplicial complex on $V$ generated by the maximal stable sets of $G$:
    $$
    \calS(G) := \gen{F \subseteq V\mid F \text{ is a maximal stable set}} = \gen{F \subseteq V\mid V\setminus F \text{ is a minimal TDS}}.
    $$
\end{definition}

Stable sets and their properties appear in the transversal theory of hypergraphs; see \cite{hypergraphs_B} for more details.

Let $\Delta$ be a simplicial complex on $V = \set{x_1,\dots,x_n}$.
The \emph{Stanley-Reisner ideal} associated with $\Delta$ is the square-free monomial ideal in $R := \kk[V]$
$$
I_\Delta := \gen{X_F \mid F \not\in \Delta}.
$$
On the other hand, let $I \subseteq R$ be a square-free monomial ideal.
Then the \emph{Stanley-Reisner (simplicial) complex} associated to $I$ is the simplicial complex on $V$
$$
\Delta_{I} := \Set{F \subseteq V}{X_F \not\in I}
$$
where we assume the product of elements in the empty set to be $1 \in R$.
The following fact is a basic property of stable complexes. 
Its proof is a routine application of Theorem~\ref{decomp of ONI}.

\begin{fact}\label{fact. SRI of RC is ONI}
Let $G = (V,E)$ be a graph.
Then $I_{\calS(G)} = \NN(G)$ and $\Delta_{\NN(G)} = \calS(G)$.
\end{fact}

The following definition and fact are standard results in the theory of simplicial complexes.

\begin{definition}\label{def. join of simp comp}
    Let $\Delta'$ and $\Delta''$ be simplicial complexes on $V'$ and $V''$, respectively, and assume that $V' \cap V'' = \emptyset$.
    The \emph{join} of $\Delta'$ and $\Delta''$ is the simplicial complex on $V' \cup V''$:
    $$
    \Delta' * \Delta'' := \Set{F' \cup F''}{F' \in \Delta',F'' \in \Delta''}.
    $$
\end{definition}

\begin{fact}[\protect{\cite{shellable_join}}]\label{fact. D'*D'' is shell <->  D',D'' are shell}
    Let $\Delta'$ and $\Delta''$ be simplicial complexes as in Definition~\ref{def. join of simp comp}.
    Then $\Delta'*\Delta''$ is shellable if and only if $\Delta'$ and $\Delta''$ are shellable.
\end{fact}

Hence, by Fact~\ref{fact. D'*D'' is shell <->  D',D'' are shell}, if the stable complex $\calS(T)$ of an WTD tree $T$ can be expressed as the join of two simplicial complexes, $\Delta' * \Delta''$, we only need to show that $\Delta'$ and $\Delta''$ are shellable to prove that $\calS(T)$ is shellable.

\begin{definition}\label{def. even, odd, odd TD set}
    Let $T = (V,E)$ be a balanced tree.
    We denote the set of \emph{even height} vertices in $T$ by
    $$
    V_{even}(T) := \Set{v \in V}{\hh(v)\mbox{ is even}}.
    $$
    Similarly, we denote the set of \emph{odd height} vertices in $T$ by
    $$
    V_{odd}(T) := \Set{v \in V}{\hh(v)\mbox{ is odd}}.
    $$
    For $D \subseteq V$, we say $D$ is an \emph{odd-TDS} if $D$ is an $V_{odd}$-TDS in $T$; see Definition~\ref{def. S-TDS, SONI}.
\end{definition}

The following fact is a consequence of Definition~\ref{def. even, odd, odd TD set} and the properties of the interior graphs.

\begin{fact}\label{fact. odd-TDS is RD-set of TB}
   Let $T = (V,E)$ be a tree and let $T_\B$ and $T_\R$ be its interior graphs.
   Then
   $$
V_{even}(T_\B) = V_\B(T_\B) \mbox{ and } V_{odd}(T_\B) = V_\R(T_\B). 
$$
Similarly, we have
$$
V_{even}(T_\R) = V_\R(T_\R) \mbox{ and } V_{odd}(T_\R) = V_\B(T_\R). 
$$
\end{fact}

\begin{definition}\label{def. even stable complex of delta trees}
    Let $T = (V,E)$ be a balanced tree.
    The \emph{even-stable complex} of $T$ is the simplicial complex on $V_{even}(T)$:
    $$
    \calS_{even}(T) := \gen{V_{even}(T)\setminus D\mid D \text{ is a minimal odd-TDS}}.
    $$
\end{definition}

The following lemma is used to prove the main theorem of this subsection.

\begin{lemma}\label{lem. vert. partitioning of T}
    Let $T$ be a tree with a 2-coloring $\chi$. Then $V(T)$ can be partitioned as
    $$
    V(T) = V_{even}(T_\B) \cup V_{even}(T_\R) \cup V_1(T).
    $$
\end{lemma}

\begin{proof}
The containment $V(T) \supseteq V_{even}(T_\B) \cup V_{even}(T_\R) \cup V_1(T)$ is clear by definition, so we only need to show the other containment.
Let $v \in V(T)$.
By symmetry, assume that $v$ is blue.
We consider two cases: $v \in V(T_\B)$ or $v \not\in V(T_\B)$.
First, suppose that $v \in V(T_\B)$.
Then by Corollary~\ref{cor. same height parity -> same color} and the construction of $V_\B$ (every leaf is blue color), $v \in V_{even}(T_\B)$.
Next, suppose that $v \not\in V(T_\B)$.
Then by the construction of $T_\B$, $v$ must be a blue support vertex in $T$ since the blue vertices in $T$ that are not in $T_\B$ are the blue support vertices of $T$. 
Thus $v \in V_1(T)$.
Thus we get $V(T) \subseteq V_{even}(T_\B) \cup V_{even}(T_\R) \cup V_1(T)$.

Next, we show that the sets $V_{even}(T_\B)$, $V_{even}(T_\R)$, and $V_1(T)$ are pairwise disjoint.
By Corollary~\ref{cor. same height parity -> same color}, vertices in $V_{even}(T_\B)$ are blue and vertices in $V_{even}(T_\R)$ are red, giving us $V_{even}(T_\B) \cap V_{even}(T_\R) = \emptyset$.
By symmetry, it suffices to show that $V_1(T) \cap V_{even}(T_\B) = \emptyset$.
Let $v \in V_1(T)$.
If $\chi(v) = \B$, then $v \not\in V(T_\B)$ by the construction of $T_\B$, and thus $v \not\in V_{even}(T_\B)$.
If $\chi(v) = \R$, then $v \not\in V_{even}(T_\B)$ since vertices in $V_{even}(T_\B)$ are blue.
\end{proof}

\begin{theorem}\label{thm. stable comp. is join of even complexes}
    Let $T = (V,E)$ be a tree.
    Then
    $\calS(T) = \calS_{even}(T_\B) * \calS_{even}(T_\R)$.
\end{theorem}

\begin{proof}
We will only show that $\mathcal{F}(\calS(T)) \subseteq \mathcal{F}(\calS_{even}(T_\B) * \calS_{even}(T_\R))$ (the reverse inclusion is proved similarly).
    Let $F \in \mathcal{F}(\calS(T))$ be a facet.
    Then there exists a minimal TDS $D \subseteq V(T)$ such that $F = V(T)\setminus D$.
    By Fact~\ref{fact. TDS iff RD-set U BD-set}, there exist an RD-set $D' \subseteq V_\B(T)$ and a BD-set $D'' \subseteq V_\R(T)$ in $T$ such that $D = D' \cup D''$.
    Then by Fact~\ref{fact. odd-TDS is RD-set of TB}, there are minimal odd-TDSs $D_1 \subseteq V_{even}(T_\B)$ in $T_\B$ and $D_2 \subseteq V_{even}(T_\R)$ in $T_\R$ such that
    $$
    D' = V_{1,\B} \cup D_1 \text{ and } D'' = V_{1,\R} \cup D_2
    $$
    where $V_{1,C} = V_1(T) \cap V_C(T)$ for all $C \in \set{\R,\B}$.
    Also, we have $V_1(T) = V_{1,\B} \cup V_{1,\R}$.
    Thus
    \begin{align*}
        F &= V(T)\setminus D\\
        &= V(T)\setminus (D' \cup D'')\\
        &= V(T) \setminus \left((V_{1,\B} \cup D_1) \cup (V_{1,\R} \cup D_2) \right)\\
        &= \left(V_{even}(T_\B) \cup V_{even}(T_\R) \right)\setminus (D_1 \cup D_2) \tag{Lemma~\ref{lem. vert. partitioning of T}}\\
        &= (V_{even}(T_\B)\setminus D_1) \cup (V_{even}(T_\R)\setminus D_2).
    \end{align*}
    By Definition~\ref{def. even stable complex of delta trees}, we have $V_{even}(T_\B)\setminus D_1 \in \calS_{even}(T_\B)$ and $V_{even}(T_\R)\setminus D_2 \in \calS_{even}(T_\R)$. 
\end{proof}

\subsection{Shellability of even stable complexes}\label{subsec. Shellability of even stable complexes}

In this subsection, we will prove that the even-stable complexes of WTD balanced trees are shellable; see Theorem~\ref{thm. even stable complex is shellable}.

\begin{definition}[\protect{\cite[Definition 5.1.11]{MR1251956}}]\label{def. shellability of simp comp}
    Let $\Delta$ be a pure simplicial complex.
    An ordering $F_1,\dots,F_m$ of facets of $\Delta$ is a \emph{shelling} if it satisfies one of the following equivalent conditions:
    \begin{itemize}
        \item[(i)] For all $i$ such that $1 < i \leq m$, the simplicial complex $\gen{F_1,\dots,F_{i-1}} \cap \gen{F_i}$ is pure of dimension $\dim(\Delta) - 1$.
        
        \item[(ii)] for all $i,j$ such that $1 \leq i < j \leq m$, there exist a vertex $v \in F_j \setminus F_i$ and an index $k \in [j-1]$ satisfying $F_j\setminus F_k = \set{v}$.
    \end{itemize}
    We say $\Delta$ is \emph{shellable} if $\Delta$ has a shelling order.
\end{definition}

The following equivalent definition of shellability can be proven using Definition~\ref{def. shellability of simp comp}.

\begin{fact}\label{fact. another def of shelling}
    Let $\Delta$ be a pure simplicial complex. Then $\Delta$ is shellable if and only if there is an ordering of facets $F_1,\dots,F_m$ such that for all $1 \leq i < j \leq m$, if $|F_i \cap F_j| < \dim(\Delta)$, then there exists $1 \leq k < j$ such that $F_i \cap F_j \subseteq F_k \cap F_j$ and $|F_k \cap F_j| = \dim(\Delta)$.
\end{fact}

Throughout, we will use Fact~\ref{fact. another def of shelling} to show that a simplicial complex is shellable.
We also use the following well-known result about shellable simplicial complexes; see \cite{joinIsShellable, MR2853065}.

\begin{fact}\label{fact. k-skeleton of simplex is shellable}
    Let $|V| = n$ and $1 \leq k \leq n$.
    Then the simplicial complex $\gen{S \subseteq V \mid |S| = k}$ is shellable.
\end{fact}

We first consider balanced trees of heights 0 and 1.

\begin{lemma}\label{lem. h(T) = 0,1 -> WTD and shellable}
    Let $T$ be a balanced tree with $\hh(T) \leq 1$.
    Then $\calS_{even}(T)$ is shellable.    
\end{lemma}

\begin{proof}
    The proof of Theorem~\ref{thm. char. of WTD delt. trees} implies that $T$ is WTD, and hence $\calS_{even}(T)$ is pure.
    If $\hh(T) = 0$, then $T$ is a graph with a single vertex, say $v$.
    Hence $\emptyset$ is its only TDS by assumption, so $\calS_{even}(T) = \gen{\set{v}}$, which is shellable.
    Suppose that $\hh(T) = 1$.
    Then $T$ is a tree with a unique support vertex $s$ and leaves $\set{l_1,\dots,l_k}$ for some $k > 1$.
    In particular, every minimal odd-TDS of $T$ is of the form $\set{l_i}$.
    Thus we get
    $$
    \calF(\calS_{even}(T)) = \Set{V_0\setminus \set{l_i}}{1 \leq i \leq k} = \Set{V' \subseteq V_0}{|V'| = k-1}.
    $$
    By the last equality above, $\calS_{even}(T)$ is shellable by Fact~\ref{fact. k-skeleton of simplex is shellable}. 
\end{proof}

Next, we consider WTD balanced trees of height 3.
We use the following convention for the labeling of vertices in WTD balanced trees of height 3.

\begin{assumptions}
    Let $T = (V,E)$ be a height-3 WTD balanced tree.
    Throughout this subsection, let $p := |V_1|$, $V_1 := \set{s_1,\dots,s_{p}}$. 
    We set $V_2 := \set{u_{1,1},\dots,u_{p,1}}$, where $s_iu_{i,1} \in E$ for all $i \in [p]$. 
    (Note that $p = |V_1| =|V_2|$ by Theorem~\ref{thm. char. of WTD trees}.)
    For all $i\in[p]$, set $k_i:= |N(s_i)|$ and write $\set{u_{i,2},\dots,u_{i,k_i}}:= N(s_i) \cap V_0$.
\end{assumptions}

\begin{example}\label{ex. labeling for an WTD delt tree of h = 3}
    Figure~\ref{fig: sec 3_2 labeling example} demonstrates a vertex labeling based on the assumptions.     
    \begin{figure}[ht]
        \centering
        \includegraphics{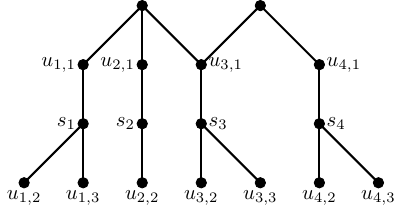}
        \caption{Vertex labeling of a height 3 WTD balanced tree}
        \label{fig: sec 3_2 labeling example}
    \end{figure}
    In this example, $p = |V_1| = |V_2| = 4$, $k_1 = |N(s_1)| = 3$, $k_2 = |N(s_2)| = 2$, $k_3 = 3$, and $k_4 = 3$.
\end{example}

To define an ordering on $\calF(\calS_{even}(T))$, we will enumerate the facets using integer vectors.
We begin with a lemma which is essentially proved in \cite[Theorem 4.11]{COURAGE_GT}.

\begin{lemma}\label{lem. |N(s) cap D| = 1}
    Let $T$ be an WTD balanced tree of height 3, and let $D$ be a minimal odd-TDS.
    Let $s \in V$ be a support vertex.
    Then $|N(s) \cap D| = 1$.
    Therefore, every minimal odd-TDS of $T$ is of the form
    $\set{u_{1,a_1},u_{2,a_2},\dots,u_{p,a_{p}}}$
where $1 \leq a_i \leq k_i$.
\end{lemma}

The following corollary is not used until Section~\ref{sec241204d}.

\begin{corollary}\label{cor. TDS = V_1 U odd TDS}
    Let $T = (V,E)$ be an WTD balanced tree of height 3.
    Then every minimal TDS of $T$ is a union of $V_1$ and a minimal odd TDS of $T$.
    Therefore by Lemma~\ref{lem. |N(s) cap D| = 1}, every minimal TDS of $T$ is of the form
    $$
    V_1 \cup \set{u_{1,a_1},u_{2,a_2},\dots,u_{p,a_{p}}}.
    $$
\end{corollary}

\begin{proof}
    Use \cite[Lemma 5.7]{COURAGE_GT}
    and Fact~\ref{fact. odd-TDS is RD-set of TB}.
\end{proof}

\begin{definition}\label{def. bijection between facets and vectors}
    Let $T$ be an WTD balanced tree of height 3.
    Define
    $$
    P_T := \Set{(a_1,\dots,a_{p})}{\text{for all } i \in [p], 1 \leq a_i \leq k_i}
    $$
    and     
    $$  
    U_T  := \Set{V_{even} \setminus \set{u_{1,a_1},\dots,u_{p,a_{p}}}}{1 \leq a_i \leq k_i}.
    $$
    Define a bijection $\nu_T: U_T  \to P_T$ by
    $$
    \nu_T(V_{even}\setminus \set{u_{1,a_1},\dots,u_{p,a_p}}) = (a_1,\dots,a_p).
    $$
    For $(a_1,\dots,a_p) \in P_T$, define
    $$
    \ell(a_1,\dots,a_p) = |\Set{i}{a_i = 1, 1 \leq i \leq p}|.
    $$
    Define a total order $<_P$ on $P_T$ as follows:
    for $\alpha,\beta \in P_T$, $\alpha <_P \beta$
    if and only if
    \begin{enumerate}
        \item $\ell(\alpha) > \ell(\beta)$, or
        \item $\ell(\alpha) = \ell(\beta)$ and $\alpha <_{lex} \beta$
    \end{enumerate}
    where $<_{lex}$ is the lexicographic order; $\alpha <_{lex} \beta$ if and only if $a_1 < b_1$ or there is an index $1 < i \leq p$ such that  $a_j = b_j$ for all $j \in [i-1]$ and $a_i < b_i$.
    Define a total order $<_T$ on $U_T $ as follows: for all $F_1,F_2 \in U_T $, 
    $F_1 <_T F_2 \text{ if and only if } \nu_T(F_1) <_P \nu_T(F_2)$.
\end{definition}

\begin{example}\label{ex. D1 D2 F1 F2 example}
Consider the sets
$$
D_1 = \set{u_{1,3},u_{2,1},u_{3,2},u_{4,1}}\ \ \  \mbox{ and }\ \ \  D_2 = \set{u_{1,1},u_{2,2},u_{3,1},u_{4,2}}
$$
in Figure~\ref{fig: two minimal odd-TDSs}.
\begin{figure}[ht]
\centering
\includegraphics{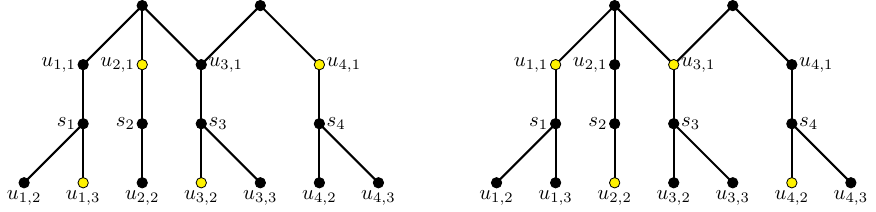}
\caption{Two minimal odd-TDSs $D_1$ (left) and $D_2$ (right)}
\label{fig: two minimal odd-TDSs}
\end{figure}
Call this tree $T$, and set $F_i := V_{even}\setminus D_i$ for $i = 1,2$.
The sets $D_1$ and $D_2$ are minimal odd-TDSs in $T$, and we have
$$
\nu_T(F_1) = (3,1,2,1)\ \ \  \mbox{ and }\ \ \  \nu_T(F_2) = (1,2,1,2).
$$
Since $\ell(1,2,1,2) = 2 = \ell(3,1,2,1)$ and $(1,2,1,2) <_{lex} (3,1,2,1)$, we get $F_2 <_T F_1$.
\end{example}

We give some useful lemmas that will be used in the proof of Theorem~\ref{thm. even stable complex is shellable}.

\begin{lemma}\label{lem. F1 cap F2 size}
    Let $T$ be an WTD balanced tree of height 3.
    Let $F_1, F_2 \in \mathcal{F}(\calS_{even}(T))$ with $\nu_T(F_1) = (a_1,\dots,a_p)$ and $\nu_T(F_2) = (b_1,\dots,b_p)$.
    Then
    $$
    |F_1 \cap F_2| = \dim(\calS_{even}(T)) + 1 - |\Set{i}{a_i \neq b_i, 1\leq i \leq p}|.
    $$
\end{lemma}

\begin{proof}
    Set $I_0 := \Set{i}{a_i = b_i, 1\leq i \leq p}$ and $I_1 := \Set{i}{a_i \neq b_i, 1\leq i \leq p}$.
    Also, write $D_i = V_{even}\setminus F_i$ for $i = 1,2$.
    Then by definition of $\nu_T$, we get
    $$
    D_1 = \set{u_{1,a_1},\dots,u_{p,a_p}} \ \ \ \text{ and } D_2 = \set{u_{1,b_1},\dots,u_{p,b_p}}.
    $$
    Thus $|D_1 \cap D_2| = |I_0|$, and $|D_2| - |I_0| = |I_1|$.
    Hence we get
    \begin{align*}
        |F_1 \cap F_2| &= |(V_{even} \setminus D_1) \cap (V_{even}\setminus D_2)|\\
        &= |V_{even}\setminus (D_1 \cup D_2)|\\
        &= |V_{even}| - |D_1 \cup D_2| \tag{$D_1 \cup D_2 \subseteq V_{even}$}\\
        &= |V_{even}| - (|D_1| + |D_2| - |D_1 \cap D_2|)\\
        &= |V_{even}| - |D_1| - |D_2| + |I_0|\\
        &= |F_1| - |D_2| + |I_0| \tag{$|F_1| = |V_{even}| - |D_1|$}\\
        &= |F_1| - |I_1| \tag{$|D_2| - |I_0| = |I_1|$}\\
        &= \dim(\calS_{even}(T)) + 1 - |\Set{i}{a_i \neq b_i, 1\leq i \leq p}|.
    \end{align*}
\end{proof}

In particular, for two facets $F_1,F_2 \in \calF(\calS_{even}(T))$, $|F_1 \cap F_2| = \dim(\calS_{even}(T))$ if and only if $\nu_T(F_1)$ and $\nu_T(F_2)$ differ by exactly one entry.

\begin{lemma}\label{lem. non-zero to 1 is TDS}
    Let $T$ be an WTD balanced tree of height 3, and let 
    $$
    D = \set{u_{1,a_1},\dots,u_{p,a_p}} \subseteq V
    $$ 
    be a minimal odd-TDS in $T$.
    Suppose that there exists some index $i$ such that $a_i \neq 1$.
    Then for any $c \in [k_i]$, the set
    $D' = (D\setminus \set{u_{i,a_i}}) \cup \set{u_{i,c}}$
    is also a minimal odd-TDS.
\end{lemma}

\begin{proof}
    First, suppose that $c \neq 1$. 
    Then $u_{i,c}$ is a leaf adjacent to $s_i$.
    So, $D'$ is obtained by replacing the leaf $u_{i,a_i}$ by another leaf $u_{i,c}$ that is adjacent to the same support vertex $s_i$. 
    Thus $D'$ is also a minimal odd-TDS.
    Next, suppose that $c = 1$.
    Since $D$ is a minimal odd-TDS, we get $N(s_i) \cap D = \set{u_{i,a_i}}$ by Lemma~\ref{lem. |N(s) cap D| = 1}. 
    Hence we have $N(D\setminus \set{u_{i,a_i}}) = V_{odd}\setminus \set{s_i}$.
    Since $N(u_{i,1}) \supseteq \set{s_i}$, we get
    $$
    N(D') = N(D\setminus \set{u_{i,a_i}}) \cup N(u_{i,1}) = V_{odd}.
    $$ 
    To show that $D'$ is minimal, consider the function $\mathcal{D'}: D' \to N(D')$ given by 
    $$
    \mathcal{D'}(v) = \begin{cases}
        s_j & v = u_{j,a_j} \text{ where } j \neq i\\
        s_i & v = u_{i,1}
    \end{cases}.
    $$
    Then Lemma~\ref{lem. |N(s) cap D| = 1} implies that $\mathcal{D'}$ is a domination selector, therefore $D'$ is minimal.
\end{proof} 

We prove the main result of this subsection.

\begin{theorem}\label{thm. even stable complex is shellable}
    Let $T$ be an WTD balanced tree of height 3.
    Then the total order $<_T$ on $U_T$ restricted to $\mathcal{F}(\calS_{even}(T))$ is a shelling of $\calS_{even}(T)$.
\end{theorem}

\begin{proof}
    Since $\mathcal{F}(\calS_{even}(T)) \subseteq U_T $ by Lemma~\ref{lem. |N(s) cap D| = 1}, $<_T$ is a total order on $\mathcal{F}(\calS_{even}(T))$.
    Now index the facets in $\calS_{even}(T)$ under $<_T$ and write
    $$
    \mathcal{F}(\calS_{even}(T)) = \set{F_1,F_2,\dots,F_m}.
    $$
    To show that $<_T$ is a shelling, suppose that there exist indicies $i,j$ with $1 \leq i < j \leq m$ such that $|F_i \cap F_j| < \dim(\calS_{even}(T))$.
    We need to find a facet $F \in \mathcal{F}(\calS_{even}(T))$ such that $F <_T F_j$, $F_i \cap F_j \subseteq F \cap F_j$, and $|F \cap F_j| = \dim(\calS_{even}(T))$.
    Write $\nu_T(F_i) = (a_1,\dots,a_p)$ and $\nu_T(F_j) = (b_1,\dots,b_p)$.
    Since $F_i <_T F_j$, $\nu_T(F_i) <_P \nu_T(F_j)$, so there are two cases to consider.\newline
    
    \noindent\emph{Case 1}: $\ell(\nu_T(F_i)) > \ell(\nu_T(F_j))$.
    Then $\nu_T(F_i)$ has more 1's then $\nu_T(F_j)$.
    Hence there exists some index $k$ such that $a_k = 1$ and $b_k \neq 1$.
    Now consider 
    $$
    F = V_{even}\setminus \set{u_{1,b_1},\dots,u_{k-1,b_{k-1}}, u_{k,1}, u_{k+1,b_{k+1}},\dots,u_{p,b_p}} \in U_T 
    $$
    with
    $$
    \nu_T(F) = (b_1,\dots,b_{k-1},1,b_{k+1},\dots,b_p).
    $$
    Then $F \in \mathcal{F}(\calS_{even}(T))$ by Lemma~\ref{lem. non-zero to 1 is TDS}.
    Since $\ell(\nu_T(F)) = \ell(\nu_T(F_j)) + 1$, we get $F <_T F_j$.
    Since we have
    $$
    F_i \cap F_j = V_{even} \setminus \left( \set{u_{1,a_1},\dots,u_{p,a_p}} \cup \set{u_{1,b_1},\dots,u_{p,b_p}}\right)
    $$
    and
    \begin{align*}
        F \cap F_j &= V_{evem} \setminus \left(\set{u_{1,b_1},\dots,u_{k-1,b_{k-1}}, u_{k,1}, u_{k+1,b_{k+1}},\dots,u_{p,b_p}} \cup \set{u_{1,b_1},\dots,u_{p,b_p}} \right)\\ 
        &= V_{even} \setminus \left( \set{u_{k,1}}\cup\set{u_{1,b_1},\dots,u_{p,b_p}}\right)
    \end{align*}
    the choice of $k$ giving $u_{k,1} = u_{k,a_k} \in \set{u_{1,a_1},\dots,u_{p,a_p}}$ implies 
    $$
    \left( \set{u_{k,1}}\cup\set{u_{1,b_1},\dots,u_{p,b_p}}\right) \subseteq \left( \set{u_{1,a_1},\dots,u_{p,a_p}} \cup \set{u_{1,b_1},\dots,u_{p,b_p}}\right)
    $$
    hence $F_i \cap F_j \subseteq F \cap F_j$. 
    Finally, we have $|F \cap F_j| = \dim(\calS_{even}(T))$ by Lemma~\ref{lem. F1 cap F2 size} as $\nu_T(F_j)$ and $\nu_T(F)$ only differ by one digit.\newline

    \noindent\emph{Case 2}: $\ell(\nu_T(F_i)) = \ell(\nu_T(F_j))$ and $\nu_T(F_i) <_{lex} \nu_T(F_j)$.
    Since $\nu_T(F_i) <_{lex} \nu_T(F_j)$, there is an index $k \in [p]$ such that $a_k < b_k$ and $a_t = b_t$ for all $t < k$.
    Let 
    $$
    F = V_{even} \setminus \set{u_{1,b_1},\dots,u_{k-1,b_{k-1}}, u_{k,a_k}, u_{k+1,b_{k+1}},\dots,u_{p,b_p}} \in U_T 
    $$ 
    with
    $$
    \nu_T(F) = (b_1,\dots,b_{k-1},a_{k},b_{k+1},\dots,b_p).
    $$
    Then $F \in \mathcal{F}(\calS_{even}(T))$ by Lemma~\ref{lem. non-zero to 1 is TDS}.
    To show that $F <_T F_j$, we consider two cases.
    If $a_k = 1$, then $\ell(\nu_T(F)) = \ell(\nu_T(F_j)) + 1$, hence $F <_T F_j$.
    If $a_k > 1$, then $\ell(\nu_T(F)) = \ell(\nu_T(F_j))$ and $\nu_T(F) <_{lex} \nu_T(F_j)$ as $a_k < b_k$, hence $F <_T F_j$.
    Similar to Case 1, we get $F_i \cap F_j \subseteq F \cap F_j$, and $|F \cap F_j| = \dim(\calS_{even}(T))$ by Lemma~\ref{lem. F1 cap F2 size}.
\end{proof}

\begin{example}
    We demonstrate Case 2 in the proof of Theorem~\ref{thm. even stable complex is shellable} using the WTD balanced tree from Example~\ref{ex. labeling for an WTD delt tree of h = 3}.
    Consider the sets $F_1$ and $F_2$ from Example~\ref{ex. D1 D2 F1 F2 example}.
    Since $F_2 <_T F_1$ and $|F_2 \cap F_1| = 3 < 6 = \dim(\calS_{even}(T))$, we construct a facet $F \in \mathcal{F}(\calS_{even}(T))$ such that $F_2 \cap F_1 \subseteq F \cap F_1$ and $|F\cap F_1| = 6$.
    Let $F \in U_T $ such that $\nu_T(F) = (1,1,2,1)$.
    Then $F \in \mathcal{F}(\calS_{even}(T))$ by Lemma~\ref{lem. non-zero to 1 is TDS}, and 
    $$
    |F \cap F_1| = \dim(\calS_{even}(T)) + 1 - 1 = \dim(\calS_{even}(T))
    $$
    by Lemma~\ref{lem. F1 cap F2 size}.
    Since 
    \begin{align*}
    F_2 \cap F_1 &= V_{even}\setminus \left(\set{u_{1,3},u_{2,1},u_{3,2},u_{4,1}} \cup \set{u_{1,1},u_{2,2},u_{3,1},u_{4,2}} \right)
    \end{align*}
    and
    $$
    F \cap F_1 = V_{even} \setminus \left(\set{u_{2,1}} \cup \set{u_{1,1},u_{2,2},u_{3,1},u_{4,2}} \right)
    $$
    we get $F_2 \cap F_1 \subseteq F \cap F_1$.
    Since $\ell(F) = 3$ and $\ell(F_2) = 2$, we get $F <_T F_2$.
\end{example}

\begin{corollary}\label{cor. even stable complex of delta forest is shellable}
    Let $T$ be an WTD balanced forest.
    Then $\calS_{even}(T)$ is shellable.
\end{corollary}

\begin{proof}
    Let $T_1,\dots,T_c$ be the connected components of $T$, which are all WTD balanced trees.
    Then Lemma~\ref{lem. h(T) = 0,1 -> WTD and shellable} and Theorem~\ref{thm. even stable complex is shellable} implies that $\calS_{even}(T_i)$ is shellable for all $i \in [c]$.
    A set $D = D_1 \cup \cdots \cup D_c$ is a minimal odd-TDS of $T$ if and only if $D_i$ is a minimal odd-TDS of $T_i$ for all $i$. 
    Thus we get
    $$
    \calS_{even}(T) = \calS_{even}(T_1) * \cdots * \calS_{even}(T_c)
    $$
    therefore $\calS_{even}(T)$ is shellable by Fact~\ref{fact. D'*D'' is shell <->  D',D'' are shell}.
\end{proof}

\begin{theorem}\label{thm. stable complex is shellable}
    Let $T$ be an WTD tree.
    Then $\calS(T)$ is shellable.
\end{theorem}

\begin{proof}
    Let $T_\R$ and $T_\B$ be the interior graphs of $T$.
    Then by Corollary~\ref{cor. even stable complex of delta forest is shellable}, $\calS_{even}(T_\R)$ and $\calS_{even}(T_\B)$ are shellable.
    By Theorem~\ref{thm. stable comp. is join of even complexes}, we get $\calS(T) = \calS_{even}(T_\R) *\calS_{even}(T_\B)$, therefore $\calS(T)$ is shellable.
\end{proof}

As a final remark, we get the following Cohen-Macaulay theorem.

\begin{corollary}\label{cor. CM of ONI and SON}
    Let $T$ be an WTD tree and let $T'$ be an WTD balanced tree.
    Then $\kk[V(T)]/\NN(T)$ and $\kk[V_{even}(T')]/\NN_{V_{odd}(T')}(T')$ are Cohen-Macaulay.
\end{corollary}

\begin{proof}
    The fact that the first quotient $\kk[V(T)]/\NN(T)$ is Cohen-Macaulay is a consequence of Theorem~\ref{thm. stable complex is shellable} and Fact~\ref{fact. SRI of RC is ONI} (see \cite[Theorem 5.1.13]{MR1251956}).
    For the second quotient, one can show that $\NN_{V_{odd}(T')}(T')$ is the Stanley-Reisner ideal of $\calS_{even}(T')$ using Theorem~\ref{thm. decomp of SON}.
\end{proof}

\begin{remarkn}
    Let $T$ be a tree.
    The \emph{suspension} of $T$ is a tree $\Sigma T$ obtained by adding a leaf to each vertex of $T$.
    It is well known that the edge ideal of $\Sigma T$, $\II(\Sigma T)$, is Cohen-Macaulay \cite{MR1031197}.
    A \emph{complete edge subdivision} of a graph $G$ is a graph $\mathcal{E}(G)$ obtained by inserting a new vertex on each edge in $G$.
    It is not hard to see that $E(\Sigma T)$ is either a $P_2$ or a height 3 WTD balanced tree.
    Moreover, we have
    $$
    \II(\Sigma T) = \NN_{V_{odd}}(\mathcal{E}(\Sigma T)).
    $$
    Hence every Cohen-Macaulay edge ideal of trees arises as an odd-open neighborhood ideal of an WTD balanced tree. 
\end{remarkn}

\section{Cohen-Macaulay Type of WTD trees}
\label{sec241204d}

In this section, we give a graph theoretic way to compute the Cohen-Macaulay type of open neighborhood ideal of WTD trees.

\subsection{Cohen-Macaulay type of balanced trees}

In this subsection, we compute the type of $\NN_{V_{\text{odd}}}(T)$ where $T$ is an WTD balanced tree.

\begin{assumptions}
In this subsection, we assume that $T = (V,E)$ is an WTD balanced tree, and set $R = \kk[V_{\text{even}}]$ and $(V_{\text{even}})R =: \mathfrak{X}$, unless otherwise stated.
Also, we write $\NN_{\text{odd}}(T) := \NN_{V_{\text{odd}}}(T)$.
\end{assumptions}

To compute $\tp(R/\NN_{\text{odd}}(T))$, we first identify a maximal $(R/\NN_{\text{odd}}(T))$-regular sequence.
First, we consider WTD balanced trees of height 0 and 1.

\begin{lemma}\label{lemma. max. reg. seq. of hh(T) < 3}
    Suppose that $\hh(T) \leq 1$.
    Let $Z \subseteq R$ be defined as followings:
    \begin{itemize}
        \item[(1)] If $\hh(T) = 0$, set $V = \set{v}$ and $Z := \set{v}$;
        \item[(2)] If $\hh(T) = 1$, set $V_1 = \set{s}$, $V_0 = \set{\ell_1,\dots,\ell_{n_0}}$ with $n_0 \geq 2$, and
        $$
        Z := \Set{\ell_{1}-\ell_j}{2 \leq j \leq n_0}.
        $$
    \end{itemize}
    Then $Z$ is a maximal $R/\NN_{\text{odd}}(T)$-regular sequence in $\mathfrak{X}$. 
\end{lemma}

\begin{proof}
    (1) If $\hh(T) = 0$, then $T$ is a graph with a single vertex, and we have $R = \kk[v]$ and $\NN_{\text{odd}}(T) = 0$.
    Thus $R/\NN_{\text{odd}}(T) \cong R = \kk[v]$ and $\set{v} = Z$ is a maximum regular sequence for $R$.

    (2) Now suppose that $\hh(T) = 1$.
    We have $R = \kk[V_{\text{even}}] = \kk[\ell_{1},\dots, \ell_{n_0}]$ and $\NN_{\text{odd}}(T) = \gen{X_{N(s)}} = \gen{\ell_{1}\cdots\ell_{n_0}}$.
    Since the m-irreducible decomposition of $\NN_{\text{odd}}(T)$ is
    $$
    \NN_{\text{odd}}(T) = \gen{\ell_{1}\cdots \ell_{n_0}} = \bigcap_{i = 1}^{n_0} \gen{\ell_{i}},
    $$
    we have $\dim(R/\NN_{\text{odd}}(T)) = n_0 - 1$.
    Then we have
    \begin{align*}
        \frac{R}{\NN_{\text{odd}}(T) + \gen{Z}} = \frac{\kk[\ell_{1},\dots, \ell_{n_0}]}{\gen{\ell_{1}\cdots\ell_{n_0}}+\gen{\Set{\ell_{1}-\ell_{j}}{2 \leq j \leq n_0}}} \cong \frac{\kk[\ell_{1}]}{\gen{\ell_{1}^{n_0}}}.
    \end{align*}
    Since the ideal $\gen{\ell_{1}^{n_0}}$ is an m-irreducible decomposition, we get that $\dim(\kk[\ell_{1}]/\gen{\ell_{1}^{n_0}}) = 1 - 1 = 0$.
    Since $R/\NN_{\text{odd}}(T)$ is Cohen-Macaulay and $|Z| = n_0 - 1$, 
    the elements in $Z$ form a maximal $(R/\NN_{\text{odd}}(T))$-regular sequence.
\end{proof}

\begin{corollary}\label{cor. depth for ht 0 and 1}
Suppose that $\hh(T) \leq 1$. Then
$$
\depth(R/\NN_{\text{odd}}(T)) = \begin{cases}
1 & \mbox{ if }\hh(T) = 0\\
|V_0|-1 & \mbox{ if } \hh(T) = 1
\end{cases}.
$$
\end{corollary}

Now we consider WTD balanced trees of height 3.
We begin with a vertex labeling that will be used throughout the proofs in this section for simplicity.

\begin{notation}\label{notation. vertex labling}
Assume $T$ has height 3.
For $\ell = 0,1,2,3,$ set $n_\ell := |V_\ell|$.
Theorem~\ref{thm. char. of WTD delt. trees} allows us to denote the vertices of $T$ as follows:
\begin{itemize}
\item[(1)] write $V_1 := \{s_1,\ \dots\ ,s_{n_1}\}$ and $V_3 := \{r_1,\ \dots\ ,r_{n_3}\}$; 
\item[(2)] for $i \in [k]$ write $u_i$ for the unique height 2 vertex adjacent to $s_i$; and
\item[(3)] write $\ell_{i,1},\ \dots\ ,\ell_{i,m_i}$ for the leaves adjacent to $s_i$ ($m_i \geq 1$).
\end{itemize}   
Since $T$ is WTD of height 3, we have $n_1 = n_2$.
\end{notation}

\begin{example}\label{ex. vertex labeling of WTD delta trees}
Using Notation~\ref{notation. vertex labling}, we can label the vertices of an WTD balanced tree $T$ of height 3 as in Figure~\ref{fig:vertex labeling for delta trees}.
\begin{figure}[ht]
    \centering
    \includegraphics{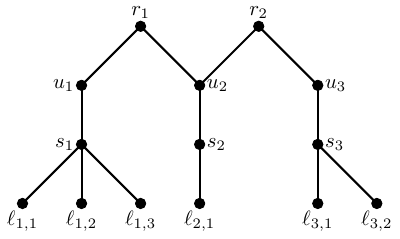}
    \caption{Example of a vertex labeling on a balanced tree using Notation~\ref{notation. vertex labling}.}
    \label{fig:vertex labeling for delta trees}
\end{figure}
Here, $n_1 = |V_1| = 3$, $n_3 = |V_3| = 2$, $m_{1} = 3$, $m_2 = 1$, and $m_3 = 2$.
The labeling rule (2) in Notation~\ref{notation. vertex labling} is well defined by Theorem~\ref{thm. char. of WTD delt. trees}.
Also, we have $|V_1| = |V_2| = n_2$ (for height 3 WTD balanced trees), $|V_3| = n_3$, and for $i \in [n_1]$, there are $m_i$ leaves adjacent to $s_i$.
Thus, we have
\begin{align}
|V_{\text{even}}| = n_2 + \sum_{i = 1}^{n_2} m_i = n_2 + n_0  
\end{align}
e.g., for the tree in Figure~\ref{fig:vertex labeling for delta trees}, we have
$$
9 = |V_{\text{even}}| = 3 + (3 + 1 + 2) = 3 + 6.
$$
\end{example}

We begin with a short lemma about the depth of $R/\NN_{\text{odd}}(T)$ when $\hh(T) = 3$.

\begin{lemma}\label{lem. dep = n_0 if hh(T) = 3}
    Let $\hh(T) = 3$.
    Using Notation~\ref{notation. vertex labling}, we have 
    $$
    \depth(R/\NN_{\text{odd}}(T)) = n_0.
    $$
\end{lemma}

\begin{proof}
    Recall that $R = \kk[V_{\text{even}}]$.
    Since $T$ is WTD, every minimal odd-TDS of $T$ has the same size; moreover, they all have size $n_2$ since the set $V_2$ is a minimal odd-TDS of $T$.
    Thus we get
    \begin{align*}
    \depth(R/\NN_{\text{odd}}(T)) &= \dim(R/\NN_{\text{odd}}(T)) \\
    &= |V_{\text{even}}| - n_2 \\
    &= \sum_{i = 1}^{n_2} m_i\\
    &= \sum_{i = 1}^{n_1} m_i \tag{$n_1 = n_2$}\\
    &= n_0.
\end{align*}
\end{proof}

Now we exhibit a maximal $R/\NN_{\text{odd}}(T)$-regular sequence when $\hh(T) = 3$.

\begin{lemma}\label{lem. max reg. seq. for height 3 D.T}
    Let $\hh(T) = 3$.
    Set
    $$
    Z := \Set{u_i - \ell_{i,j}}{1 \leq i \leq n_1, 1 \leq j \leq m_i} \subseteq R
    $$
    Then $Z$ is a maximal $R/\NN_{\text{odd}}(T)$-regular sequence in $\mathfrak{X}$.
\end{lemma}

\begin{proof}
Set $R' := \kk[u_1,\dots,u_{n_2}]$.
Then we have
\begin{align*}
\frac{R}{\NN_{\text{odd}}(T) + \gen{Z}} &\cong \frac{R'}{\left(u_1^{|N(s_1)|}, \dots ,u_{n_2}^{|N(s_{n_1})|}, X_{N(r_1)}, \dots ,X_{N(r_{n_3})}\right)R'} \tag{\textasteriskcentered{}}
\end{align*}
because every leaf $\ell_{i,j}$ is identified with its closest height 2 vertex $u_i$ and $n_1 = |V_1| = |V_2| = n_2$ since $T$ is an WTD balanced tree of height 3. 
Set 
$$
\mathfrak{I} := \left(u_1^{|N(s_1)|}, \dots ,u_{n_1}^{|N(s_{n_1})|}, X_{N(r_1)}, \dots ,X_{N(r_{n_3})}\right)R'.
$$
Since $\mathfrak{I}$ contains positive powers of all the $u_i$'s, and the $X_{N(r_i)}$ are non-unit monomials, we get $\mrad(\mathfrak{I}) = (u_1,\dots,u_{n_1})R'$.
This implies that $\dim(R'/\mathfrak{I}) = 0$.
Since $R/\NN_{\text{odd}}(T)$ is Cohen-Macaulay of depth $n_0$ by Corollary~\ref{cor. CM of ONI and SON} and Lemma~\ref{lem. dep = n_0 if hh(T) = 3}, the condition
$$
|Z| = \sum_{i = 1}^{n_2}m_i = n_0,
$$
implies that $Z$ is a maximal regular sequence.
\end{proof}

\begin{example}\label{ex. Z set example}
    Let $T$ be the tree from Example~\ref{ex. vertex labeling of WTD delta trees} and let $R = \kk[V_{\text{even}}]$; see Figure~\ref{fig: vert. lab. tree with some coloring}, some vertices are colored for guidance.
\begin{figure}[ht]
    \centering
    \includegraphics{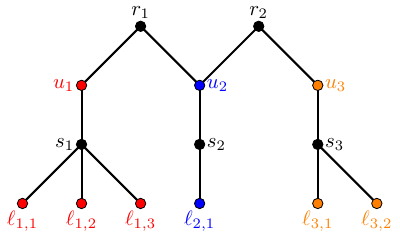}
    \caption{Tree $T$ from Example~\ref{ex. vertex labeling of WTD delta trees}}
    \label{fig: vert. lab. tree with some coloring}
\end{figure}    
    Then the maximal $R/\NN_{\text{odd}}(T)$-regular sequence from Lemma~\ref{lem. max reg. seq. for height 3 D.T} is
    \begin{align*}
            Z = &\set{\tcr{u_1 - \ell_{1,1}},\tcr{u_1 - \ell_{1,2}},\tcr{u_1 - \ell_{1,3}},\tcb{u_2 - \ell_{2,1}},\tco{u_3 - \ell_{3,1}},\tco{u_3 - \ell_{3,2}}}.    
    \end{align*}
    We have
    $$
    \NN_{\text{odd}}(T) = \gen{\tcr{u_1\ell_{1,1}\ell_{1,2}\ell_{1,3}}, \tcb{u_2\ell_{2,1}}, \tco{u_3\ell_{3,1}\ell_{3,2}},\tcr{u_1}\tcb{u_2},\tcb{u_2}\tco{u_3}}
    $$
    hence
    \begin{align*}
        \frac{R}{\NN_{\text{odd}}(T) + \gen{Z}}  &= \frac{\kk[\tcr{u_1},\tcb{u_2},\tco{u_3},\tcr{\ell_{1,1}},\tcr{\ell_{1,2}},\tcr{\ell_{1,3}},\tcb{\ell_{2,1}},\tco{\ell_{3,1}},\tco{\ell_{3,2}}]}{\gen{\tcr{u_1\ell_{1,1}\ell_{1,2}\ell_{1,3}}, \tcb{u_2\ell_{2,1}}, \tco{u_3\ell_{3,1}\ell_{3,2}},\tcr{u_1}\tcb{u_2},\tcb{u_2}\tco{u_3}}+\gen{Z}}\\\\
        &= \frac{\kk[\tcr{u_1},\tcb{u_2},\tco{u_3}]}{\gen{\tcr{u_1}^4,\tcb{u_2}^2,\tco{u_3}^3,\tcr{u_1}\tcb{u_2},\tcb{u_2}\tco{u_3}}}.
    \end{align*}
    The ideal $\gen{\tcr{u_1}^4,\tcb{u_2}^2,\tco{u_3}^3,\tcr{u_1}\tcb{u_2},\tcb{u_2}\tco{u_3}}$ is the ideal $\mathfrak{J}$ from the proof of Lemma~\ref{lem. max reg. seq. for height 3 D.T}. 
\end{example}

Now we compute the type of $R/\NN_{\text{odd}}(T)$ when $T$ is an WTD balanced tree.
We first take care of the cases when $\hh(T) < 3$.

\begin{theorem}\label{thm. type when hh(T) < 3}
    If $\hh(T) < 3$, then $\tp(R/\NN_{\text{odd}}(T)) = 1$.
\end{theorem}

\begin{proof}
    First, let $\hh(T) = 0$; hence we write $V = \set{v}$.
    Then $R = \kk[v]$ and $\NN_{\text{odd}}(T) = 0$.
    Hence we have
    \begin{align*}
        \tp(R/\NN_{\text{odd}}(T)) &= \tp(R) \tag{$\NN_{\text{odd}}(T) = 0$}\\
        &= \tp(R/\gen{v}) = 1.
    \end{align*}
    Now suppose that $\hh(T) = 1$. 
    Let $Z$ be the maximal $R/\NN_{\text{odd}}(T)$-regular sequence from Lemma~\ref{lemma. max. reg. seq. of hh(T) < 3}.
    Then we have
    \begin{align*}
        \tp(R/\NN_{\text{odd}}(T)) &= \tp \left(\frac{R}{\NN_{\text{odd}}(T) + \gen{Z}} \right) \\
        &= \tp(\kk[\ell_{1,1}]/\gen{\ell_{1,1}^{m_1}}) \tag{Lemma~\ref{lemma. max. reg. seq. of hh(T) < 3} (1)}\\
        &= 1 
    \end{align*}
as desired.
\end{proof}

Now we consider $T$ with $\hh(T) = 3$.
The ideal $\mathfrak{I}$ from Lemma~\ref{lem. max reg. seq. for height 3 D.T} has a parametric decomposition in $\kk[u_1,\dots,u_{n_2}]$ since $\mrad(\mathfrak{I}) = \gen{u_1,\dots,u_{n_2}}$.
Lemma~\ref{lem. parametric decomp of I} gives us the decomposition of $\mathfrak{I}$ explicitly.

\begin{lemma}\label{lem. parametric decomp of I}
Assume that $\hh(T) = 3$.
Let $\mathfrak{J} \leq R' = \kk[u_1,\dots,u_{n_2}]$ be the ideal from the proof of Lemma~\ref{lem. max reg. seq. for height 3 D.T}.
Set 
$$
U := \left(u_1^{|N(s_1)|},\ \dots\ , u_{n_2}^{|N(s_{n_1})|}\right)R'.
$$
The irredundant parametric decomposition of $\mathfrak{I}$ in $R'$ is
$$
\mathfrak{I} = \bigcap_{D\ min} \left(\gen D + U\right)
$$
where the intersection is taken over all minimal $V_3$-TDSs of $T$.
\end{lemma}

\begin{proof}
Notice that the ideal $\left( X_{N(r_1)},\dots,X_{N(r_{n_3})} \right)R'$ is the $V_3$-open neighborhood ideal of $T$ since $V_3 = \set{r_1,\dots,r_{n_3}}$.
Thus by Theorem~\ref{thm. decomp of SON}, we get
$$
\left( X_{N(r_1)},\dots,X_{N(r_{n_3})} \right)R' = \bigcap_{D\ min} \gen D
$$
where the intersection is taken over all minimal $V_3$-TDS of $T$.
Thus we get
$$
\mathfrak{I} = \left( X_{N(r_1)},\dots,X_{N(r_{n_3})} \right)R'+ U = \left(\bigcap_{D\ min} \gen D\right) + U =  \bigcap_{D\ min} (\gen D+U)
$$
as desired.
\end{proof}

\begin{example}\label{ex. decomp. of J}
Let $T = (V,E)$ be the tree from Example~\ref{ex. vertex labeling of WTD delta trees} and let $\mathfrak{J}$ be the ideal from Example~\ref{ex. Z set example}.
Setting 
$$
U = \gen{u_1^{|N(s_1)|},u_2^{|N(s_2)|},u_3^{|N(s_3)|}} = \gen{u_1^{4},u_2^2,u_3^3},
$$
we get
\begin{align*}
    \mathfrak{J} &= \gen{u_1^4,u_2^2,u_3^3,u_1u_2,u_2u_3}\\ 
    &= \gen{u_1^4,u_2^2,u_3^3, u_1,u_3} \cap \gen{u_1^4,u_2^2,u_3^3,u_2}\\
    &= \left(\gen{u_1,u_3} + U\right) \cap \left(\gen{u_2} + U \right).
\end{align*}
Sets $\set{u_1,u_3}$ and $\set{u_2}$ are minimal $V_3$-TDSs of $T$.
\end{example}

Now we can describe the $\tp(R/\NN_{\text{odd}}(T))$ for an WTD balanced tree of any height using minimal $V_3$-TDSs.

\begin{theorem}\label{thm. type of WTD delta trees}
    Let $R = \kk[V_{\text{even}}]$.
    Then
    $$
    \tp(R/\NN_{\text{odd}}(T)) = ``\mbox{number of minimal $V_3$-TDSs of } T."
    $$
\end{theorem}

\begin{proof}
    If $\hh(T) < 3$, then the empty set is the unique minimal $V_3$-TDS, so
    \begin{align*}
    \tp(R/\NN_{\text{odd}}(T)) &= 1 \tag{Theorem~\ref{thm. type when hh(T) < 3}}\\
    &= |\set{\emptyset}|\\
    &= ``\mbox{number of minimal $V_3$-TDSs of } T."
    \end{align*}
    Now suppose that $\hh(T) = 3$.
    Let $Z$, $R'$, and $\mathfrak{I}$ as in Lemma~\ref{lem. max reg. seq. for height 3 D.T}.
    Then we have
    \begin{align*}
        \tp(R/\NN_{\text{odd}}(T)) = \tp\left(\frac{R}{\NN_{\text{odd}}(T) + \gen{Z}}\right) \hspace{-5cm}\\
        &= \tp \left(R'/\mathfrak{I} \right) \tag{(\textasteriskcentered{}) in Lemma~\ref{lem. max reg. seq. for height 3 D.T}}\\
        &= ``\mbox{number of ideals in the parametric decomp. of }\mathfrak{I}" \\
        &= ``\mbox{number of minimal $V_3$-TDSs of } T" 
    \end{align*}
as desired.
\end{proof}

\begin{example}
    Let $T = (V,E)$ be the tree from Example~\ref{ex. vertex labeling of WTD delta trees}.
    Then by Theorem~\ref{thm. type of WTD delta trees} with Example~\ref{ex. decomp. of J}, we get
    $$
    \tp(\kk[V_{\text{even}}]/\NN_{\text{odd}}(T)) = ``\mbox{number of minimal $V_3$-TDS of } T" = 2.
    $$
\end{example}

\subsection{Cohen-Macaulay Type of WTD Trees}

\begin{assumptions}
Let $T = (V,E)$ be an WTD tree and let $T_\R$ and $T_\B$ be the two interior trees (forests) of $T$ derived from a 2-coloring $\chi: V \to \set{\R,\B}$.
Set $ N_\B(v) := N_{T_\B}(v)$ and $N_\R(v) := N_{T_\R}(v)$ for all $v \in V$.    
\end{assumptions}

Theorem~\ref{Decomposition of ONI into three ideals} gives a connection between open neighborhood ideal of $T$ and the odd open neighborhood ideals of its interior graphs.

\begin{theorem}\label{Decomposition of ONI into three ideals}
We have
\begin{align*}
\NN(T) &= \NN_{\text{odd}}(T_\B)R + \NN_{\text{odd}}(T_\R)R + \gen{V_1(T)}\\ 
&= \NN_{\text{odd}}(T_\B)R + \NN_{\text{odd}}(T_\R)R + (V_1(T))R.
\end{align*}
\end{theorem}

\begin{proof}
($\subseteq$) It suffices to show that $X_{N_T(v)} \in \NN_{\text{odd}}(T_\B)R + \NN_{\text{odd}}(T_\R)R + \gen{V_1(T)}$ for all $v \in V$.
Let $v \in V$.
By symmetry, assume that $\chi(v) = R$.

\setlength{\leftskip}{5mm}

\noindent\emph{Case 1:} $v \not\in V(T_\B)$.  
Then by Definition~\ref{def. interior graphs}, there exists some blue support vertex $s \in V(T)$ that is adjacent to $v$.
Suppose that $v$ is a leaf in $T$. 
Then $N_T(v) = \set{s}$.
Since $s \in \gen{V_1(T)}$, we have $X_{N_T(v)} = s \in \gen{V_1(T)}$.
Now suppose that $v$ is not a leaf.
Since $s$ is a support vertex, there exists some leaf $\ell \in V(T)$ that is adjacent to $s$.
Then $X_{N_T(\ell)} = s$ divides $X_{N_T(v)}$ since $s \in N_T(v)$.
Hence $X_{N_T(v)}$ is a redundant generator of $\NN(T)$.

\vspace{2mm}

\noindent\emph{Case 2:} $v \in V(T_\B)$.
Then $N_{T_\B}(v) \subseteq N_T(v)$.
Thus $X_{N_\B(v)}|X_{N_T(v)}$, so $X_{N_T(v)} \in \NN(T_\B)$.

\setlength{\leftskip}{0mm}

\vspace{2mm}

\noindent($\supseteq$) Every support vertex of $T$ is a generator of $\NN(T)$, being the open neighborhood of each leaf it is adjacent to, so we get $\gen{V_1(T)} \subseteq \NN(T)$.
To complete the proof, by symmetry, we show that $\NN(T) \supseteq \NN_{\text{odd}}(T_\B)$.
Note that all we need to show is that $X_{N_\B(v)} \in \NN(T)$ for all $v \in V_{\text{odd}}(T_\B)$ as this shows that the generators of $\NN_{\text{odd}}(T_\B)$ are in $\NN(T)$.
So, let $v \in V_{\text{odd}}(T_\B)$.
We show that $N_\B(v) = N_T(v)$ which will imply that $X_{N_\B(v)} = X_{N_T(v)} \in \NN(T)$. 
To this end, assume by way of contradiction that $N_\B(v) \subsetneq N_T(v)$.
Since $\hh(v) \in \{1,3\}$ by the definition of $V_{\text{odd}}(T_\B)$ along with Theorem~\ref{thm. char. of WTD delt. trees} and \ref{thm. char. of WTD trees}, we get $\chi(v) = \R$ as $v \in V(T_\B)$.
Since $N_\B(v) \subsetneq N_T(v)$, there exists a vertex $u \in V$ such that $u \in N_T(v)\setminus V(T_\B)$ and $uv \in E(T)$.
Since $\chi(v) = R$, we have $\chi(u) = B$.
Since $u \not\in V(T_\B)$, $u$ must be a vertex that is deleted from $T$ while constructing $T_\B$.
By definition of $T_\B$, $u$ must be a support vertex since $\chi(u) = B$.
But, this implies that $v \not\in V(T_\B)$ since $v$ must be deleted from $T$ while constructing $T_\B$ as $v \in N_T(u)$, a contradiction.
\end{proof}

\begin{example}
    Let $T$, $T_\B$, and $T_\R$ be the graphs from Example~\ref{ex. decomp. of a tree} and set $R = \kk[V]$.
    Then we have
    \begin{align*}
        \NN_{\text{odd}}(T_\B) &= \gen{\tcb{\ell_3u_3}, \tcb{\ell_4u_4},\tcb{\ell_5u_5},\tcb{u_3u_4},\tcb{u_4u_5}},\\
        \NN_{\text{odd}}(T_\R) &= \gen{\tcr{\ell_1u_1},\tcr{\ell_2u_2},\tcr{u_1u_2}},\\
        (V_1(T))R &= \gen{\tcb{s_1},\tcb{s_2},\tcr{s_3},\tcr{s_4},\tcr{s_5}}.
    \end{align*}
    The variables are colored using the same coloring of $T$ from Example~\ref{ex. decomp. of a tree}.
    The open neighborhood ideal of $T$ is given by
    \begin{align*}
        \NN(T) &= \langle s_1,s_2,s_3,s_4,s_5,\ell_1u_1,\ell_2u_2,\ell_3u_3,\ell_4u_4,\ell_5u_5,s_1r_1,s_2u_3r_1,s_3u_2r_2,s_4r_2r_3,\\
        &\ \ \ \ \ s_5r_3,u_1u_2,u_3u_4,u_4u_5\rangle\\
        &= \gen{\tcb{s_1},\tcb{s_2},\tcr{s_3},\tcr{s_4},\tcr{s_5},\tcr{\ell_1u_1},\tcr{\ell_2u_2},\tcb{\ell_3u_3},\tcb{\ell_4u_4},\tcb{\ell_5u_5},\tcr{u_1u_2},\tcb{u_3u_4},\tcb{u_4u_5}}.
    \end{align*}
    Notice that the minimal set of generators of $\NN(T)$ is exactly the union of the generating sets for $\NN_{\text{odd}}(T_\B)$, $\NN_{\text{odd}}(T_\R)$, and $(V_1(T))R$.
    Thus we get
    $$
    \NN(T) = \NN_{\text{odd}}(T_\B)+\NN_{\text{odd}}(T_\R)+(V_1(T))R.
    $$
\end{example}

By the construction of the interior graphs $T_\B$ and $T_\R$, we have $\NN_{\text{odd}}(T_\B) \subseteq (\chi^{-1}(B) \setminus V_1(T))R$ and $\NN_{\text{odd}}(T_\R) \subseteq (\chi^{-1}(R)\setminus V_1(T))R$; that is, the variables of the generators of $\NN_{\text{odd}}(T_\B)$ and $\NN_{\text{odd}}(T_\R)$ are blue and red, respectively.
Hence the ideals $\NN_{\text{odd}}(T_\B)$, $\NN_{\text{odd}}(T_\R)$, and $(V_1(T))R$ use pairwise disjoint sets of variables for their generators.

The following fact is used for Corollary~\ref{cor. CMness and depth of Delta-forests} which shows that the odd open neighborhood ideal of an WTD balanced forest is also Cohen-Macaulay. 
The depth computation is used in Theorem~\ref{thm. max. reg. seq. and type of Delta-forest}.

\begin{fact}\label{fact. delta forest are CM}
    Let $T_1$ and $T_2$ be WTD balanced trees with $V(T_1) \cap V(T_2) = \emptyset$.
    Let $T = (V(T_1) \cup V(T_2), E(T_1) \cup E(T_2))$.
    Then $\kk[V_{\text{even}}(T)]/\NN_{\text{odd}}(T)$ is Cohen-Macaulay over any field, and
    $$
    \depth\left(\frac{\kk[V_{\text{even}}(T)]}{\NN_{\text{odd}}(T)}\right) = \depth\left(\frac{\kk[V_{\text{even}}(T_1)]}{\NN_{\text{odd}}(T_1)}\right) + \depth\left(\frac{\kk[V_{\text{even}}(T_2)]}{\NN_{\text{odd}}(T_2)}\right).
    $$
\end{fact}

\begin{corollary}\label{cor. CMness and depth of Delta-forests}
    Let $T$ be an WTD balanced forest with connected components $T_1$,$\dots$,$T_c$.
    Then $\kk[V_{\text{even}}(T)]/\NN_{\text{odd}}(T)$ is Cohen-Macaulay and
    $$
    \depth\left(\frac{\kk[V_{\text{even}}(T)]}{\NN_{\text{odd}}(T)}\right) = \sum_{i = 1}^c \depth\left(\frac{\kk[V_{\text{even}}(T_i)]}{\NN_{\text{odd}}(T_i)}\right).
    $$
\end{corollary}

\begin{proof}
    Induct on $c$, and apply Lemma~\ref{fact. delta forest are CM}.
\end{proof}

Next, we state a result on a maximal $\kk[V_{\text{even}}]/\NN_{\text{odd}}(T)$-regular sequence of an WTD balanced forest $T$ that is used in Theorem~\ref{thm. CM type of WTD trees}.

\begin{theorem}\label{thm. max. reg. seq. and type of Delta-forest}
    Let $T$ be an WTD balanced forest with connected components $T_1,\dots,T_c$. 
    Let $R = \kk[V_{\text{even}}]$ and $R_i = \kk[V_{\text{even}}(T_i)]$ for all $i \in [c]$.
    For each $i \in [c]$, using Notation~\ref{notation. vertex labling} for each $T_i$, let $Z_i$ be the maximal regular sequence constructed in Lemmas~\ref{lemma. max. reg. seq. of hh(T) < 3} and~\ref{lem. max reg. seq. for height 3 D.T}.
    Set $Z = \bigcup_{i = 1}^c Z_i$. 
    Then 
    \begin{itemize}
        \item[(1)]$Z$ is a maximal $R/\NN_{\text{odd}}(T)$-regular sequence in $\gen{V_{\text{even}}}$.
        \item[(2)] $\tp(R/\NN_{\text{odd}}(T)) =``\mbox{number of minimal $V_3$-TDSs in }T."$
    \end{itemize}
\end{theorem}

\begin{proof}
Since $T$ is a forest with connected components $T_1,\dots,T_c$, we have $V_{\text{odd}} = \bigcup_{i = 1}^c V_{\text{odd}}(T_i)$ and $V_{\text{even}} = \bigcup_{i = 1}^c V_{\text{even}}(T_i)$.
So, we get
\begin{align*}
    \NN_{\text{odd}}(T) &= \gen{\Set{X_{N_T(v)}}{v \in V_{\text{odd}}(T)}}\\
    &= \gen{\bigcup_{i = 1}^c \Set{X_{N_{T_i}(v)}}{v \in V_{\text{odd}}(T_i)}}\\
    &= \sum_{i = 1}^c \gen{\Set{X_{N_{T_i}(v)}}{v \in V_{\text{odd}}(T_i)}}\\
    &= \sum_{i = 1}^c \NN_{\text{odd}}(T_i).
\end{align*}
Hence we have
\begin{align*}
    \frac{R}{\NN_{\text{odd}}(T)+\gen{Z}} &= \frac{R}{\sum_{i = 1}^c \NN_{\text{odd}}(T_i)+\sum_{j = 1}^c \gen{Z_j}}\\
    &= \frac{R}{\sum_{i = 1}^c \left(\NN_{\text{odd}}(T_i) + \gen{Z_i} \right)}\\
    &= \bigotimes_{i = 1}^c \frac{R_i}{\NN_{\text{odd}}(T_i) + \gen{Z_i}} 
\end{align*}
where the tensor product is taken over $\kk$.
Since $\dim\left(R_i/(\NN_{\text{odd}}(T_i) + \gen{Z_i})\right) = 0$ for all $i$, we get $\dim(R/(\NN_{\text{odd}}(T) + \gen{Z})) = 0$.
Note that we have
\begin{align*}
    \depth\left(\frac{\kk[V_{\text{even}}(T)]}{\NN_{\text{odd}}(T)}\right) &= \sum_{i = 1}^c\depth\left(\frac{\kk[V_{\text{even}}(T_i)]}{\NN_{\text{odd}}(T_i)}\right) \tag{Corollary~\ref{cor. CMness and depth of Delta-forests}}\\
    &= \sum_{i = 1}^c |Z_i| \tag{definition of $Z_i$}\\
    &= |Z|.
\end{align*}
So, $\kk[V_{\text{even}}(T)]/\NN_{\text{odd}}(T)$ is Cohen-Macaulay (Corollary~\ref{cor. CMness and depth of Delta-forests}) of depth $|Z|$, hence $Z$ is a maximal regular sequence.

Noting that $R \cong \otimes_{i = 1}^c R_i$ (tensoring over $\kk$), the type of $R/\NN_{\text{odd}}(T)$ is given by
\begin{align*}
    \tp\left(\frac{R}{\NN_{\text{odd}}(T)}\right)  &= \tp\left(\bigotimes_{i = 1}^c \frac{R_i}{\NN_{\text{odd}}(T_i)}\right)\\
    &= \dim_\kk\left(\Ext^0_{\otimes_{i = 1}^c R_i}\left(\kk,\bigotimes_{i = 1}^c \frac{R_i}{\NN_{\text{odd}}(T_i)}\right) \right)\\
    &= \dim_\kk\left(\bigotimes_{i = 1}^c \Ext_{R_i}^0\left(\kk,\frac{R_i}{\NN_{\text{odd}}(T_i)}\right) \right)\\
    &= \prod_{i = 1}^c \dim_\kk\left(\Ext_{R_i}^0\left( \kk,\frac{R_i}{\NN_{\text{odd}}(T_i)}\right) \right)\\
    &= \prod_{i = 1}^c\tp(R_i/\NN_{\text{odd}}(T_i)).
\end{align*}
By Theorem~\ref{thm. type of WTD delta trees}, $\tp(R_i/\NN_{\text{odd}}(T_i))$ is given by the number of minimal $V_3(T_i)$-TDSs of $T_i$ for each $i$.
Since $T$ is the union of the connected components $T_1,\dots,T_c$, the number of minimal $V_3$-TDS of $T$ is the product of the number of minimal $V_3(T_i)$-TDSs of $T_i$ for all $i$.
Thus we get
\begin{align*}
\tp\left(\frac{R}{\NN_{\text{odd}}(T)}\right) &= \prod_{i = 1}^c\tp(R_i/\NN_{\text{odd}}(T_i))\\
&= ``\mbox{number of minimal $V_3$-TDSs in }T".
\end{align*}
\end{proof}

\begin{example}\label{ex. type of T_B, T_B = T_R}
    We demonstrate the computations in Theorem~\ref{thm. max. reg. seq. and type of Delta-forest}.
    Consider $T_\B$ from Example~\ref{ex. T, T_B = T_R} shown below in Figure~\ref{fig: T_B from T_B = T_R}.
    \begin{figure}[ht]
        \centering
        \includegraphics{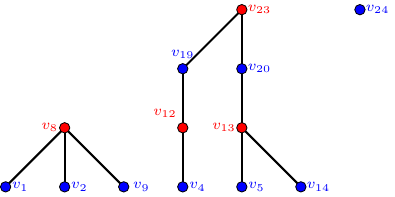}
        \caption{$T_\B$ from Example~\ref{ex. T, T_B = T_R}}
        \label{fig: T_B from T_B = T_R}
    \end{figure}
    For each connected component, labeling them as $T_1$, $T_2$, and $T_3$ from left to right, we get
    \begin{align*}
        \frac{R_1}{\NN_{\text{odd}}(T_1)} &= \frac{\kk[v_1,v_2,v_9]}{\gen{v_1v_2v_9}}\\\\
        \frac{R_2}{\NN_{\text{odd}}(T_2)} &= \frac{\kk[v_4,v_5,v_{14},v_{19},v_{20}]}{\gen{v_4v_{19},v_5v_{14}v_{20},v_{19}v_{20}}}\\\\
        \frac{R_3}{\NN_{\text{odd}}(T_3)} &= \frac{\kk[v_{24}]}{\gen{v_{24}}}. 
    \end{align*}
    Also, we have
    $$
    \frac{R}{\NN_{\text{odd}}(T_\B)} = \frac{\kk[v_1,v_2,v_9,v_4,v_5,v_{14},v_{19},v_{20},v_{24}]}{\gen{v_1v_2v_9,v_4v_{19},v_5v_{14}v_{20},v_{19}v_{20}, v_{24}}}.
    $$
    Thus we get
    \begin{align*}
        \frac{R}{\NN_{\text{odd}}(T_\B)} &= \frac{\kk[v_1,v_2,v_9,v_4,v_5,v_{14},v_{19},v_{20},v_{24}]}{\gen{v_1v_2v_9,v_4v_{19},v_5v_{14}v_{20},v_{19}v_{20}, v_{24}}}\\
        &= \frac{\kk[v_1,v_2,v_9]}{\gen{v_1v_2v_9}} \otimes_\kk \frac{\kk[v_4,v_5,v_{14},v_{19},v_{20}]}{\gen{v_4v_{19},v_5v_{14}v_{20},v_{19}v_{20}}} \otimes_\kk \frac{\kk[v_{24}]}{\gen{v_{24}}}\\
        &= \frac{R_1}{\NN_{\text{odd}}(T_1)} \otimes_\kk \frac{R_2}{\NN_{\text{odd}}(T_2)} \otimes_\kk \frac{R_3}{\NN_{\text{odd}}(T_3)}.
    \end{align*}
    The only minimal $V_3$-TDS for $T_1$ and $T_3$ is the empty set since $V_3(T_1) = \emptyset = V_3(T_3)$, and the minimal $V_3$-TDSs for $T_2$ are $\set{v_{19}}$ and $\set{v_{20}}$.
    Thus by Theorem~\ref{thm. type of WTD delta trees}, we get
    $$
    \tp(R_1/\NN_{\text{odd}}(T_1)) = 1\ \ \ \ \ \ \tp(R_2/\NN_{\text{odd}}(T_2)) = 2\ \ \ \ \ \ \tp(R_3/\NN_{\text{odd}}(T_3)) = 1.
    $$
    Therefore, the type of $R/\NN_{\text{odd}}(T_\B)$ is
    $$
    \tp(R/\NN_{\text{odd}}(T_\B)) = \tp(R_1/\NN_{\text{odd}}(T_1)) \cdot \tp(R_2/\NN_{\text{odd}}(T_2)) \cdot \tp(R_3/\NN_{\text{odd}}(T_3)) = 2.
    $$
\end{example}

Here is the main result of this thesis.
It shows how to compute the Cohen-Macaulay type of $R/\NN(T)$ for WTD trees using only graph-theoretic information about $T$.

\begin{theorem}\label{thm. CM type of WTD trees}
    Let $R = \kk[V]$, $R_\B = \kk[V_{\text{even}}(T_\B)]$, and $R_\R = \kk[V_{\text{even}}(T_\R)]$.
    Let $m_\B$ and $m_\R$ be the numbers of minimal $V_3$-TDSs in $T_\B$ and $T_\R$, respectively.
    The Cohen-Macaulay type of $R/\NN(T)$ is
    \begin{align*}
        \tp(R/\NN(T)) &= \tp(R_\B/\NN_{\text{odd}}(T_\B)) \cdot \tp(R_\R/\NN_{\text{odd}}(T_\R))\\
            &= m_\B \cdot m_\R.
    \end{align*}
\end{theorem}

\begin{proof}
    Let $Z_\B$ and $Z_\R$ be the maximal regular sequences for $\kk[V_{\text{even}}(T_\B)]/\NN_{\text{odd}}(T_\B)$ and $\kk[V_{\text{even}}(T_\R)]/\NN_{\text{odd}}(T_\R)$ computed in Theorem~\ref{thm. max. reg. seq. and type of Delta-forest}, respectively, and let $Z = Z_\B \cup Z_\R$.
    Then we have 
   \begin{align*}
        \frac{R}{\NN(T)+\gen{Z}} &= \frac{R}{\NN_{\text{odd}}(T_\B) + \NN_{\text{odd}}(T_\R) + \gen{V_1(T)}+\gen{Z}} \tag{Theorem~\ref{Decomposition of ONI into three ideals}}\\
        &= \frac{\kk[V_{\text{even}}(T_\B) \cup V_{\text{even}}(T_\R) \cup V_1(T)]}{\NN_{\text{odd}}(T_\B) + \NN_{\text{odd}}(T_\R) + \gen{V_1(T)}+\gen{Z}} \tag{Lemma~\ref{lem. vert. partitioning of T}}\\
        &\cong \frac{\kk[V_{\text{even}}(T_\B) \cup V_{\text{even}}(T_\R)]}{\NN_{\text{odd}}(T_\B) + \NN_{\text{odd}}(T_\R)+\gen{Z}}\\
        &= \frac{\kk[V_{\text{even}}(T_\B)]}{\NN_{\text{odd}}(T_\B)+\gen{Z_\B}} \otimes_\kk \frac{\kk[V_{\text{even}}(T_\R)]}{\NN_{\text{odd}}(T_\R) + \gen{Z_\R}}.
    \end{align*}
    Note that we have 
    $$
    \dim\left(\frac{\kk[V_{\text{even}}(T_\B)]}{\NN_{\text{odd}}(T_\B) + \gen{Z_\B}}\right) =0= \dim\left(\frac{\kk[V_{\text{even}}(T_\R)]}{\NN_{\text{odd}}(T_\R) + \gen{Z_\R}}\right)
    $$
    by Theorem~\ref{thm. max. reg. seq. and type of Delta-forest}.
    Hence we get  
    $$
    \dim\left(\frac{R}{\NN(T)+\gen{Z}}\right) = 0
    $$
    so
    $$
    \depth\left(\frac{R}{\NN(T)+\gen{Z}}\right) = 0.
    $$
    
    Let $\mathcal{D}$ be the set of minimal TDSs of $T$, let $\mathcal{D}_\B$ and $\mathcal{D}_\R$ be the sets of minimal odd TDSs of $T_\B$ and $T_\R$, respectively.
    Then \cite[Corollary 2.9]{COURAGE_GT} and \ref{cor. TDS = V_1 U odd TDS} give us that $D \in \mathcal{D}$ if and only if $D = D_\B \cup D_\R \cup V_1(T)$ for some $D_\B \in \mathcal{D}_\B$ and $D_\R \in \mathcal{D}_\R$; note that the sets $D_\B$, $D_\R$, and $V_1(T)$ are pairwise disjoint by Lemma~\ref{lem. vert. partitioning of T} since $D_\B \subseteq V_{\text{even}}(T_\B)$ and $D_\R \subseteq V_{\text{even}}(T_\R)$.
    Fix $D \in \mathcal{D}$, $D_\B \in \mathcal{D}_\B$, and $D_\R \in \mathcal{D}_\R$ such that $D = D_\B \cup D_\R \cup V_1(T)$.
    Then we have
    \begin{align*}
        \dim\left(\frac{R}{\NN(T)} \right) &= |V(T)| - |D|\\
        &= |V_{\text{even}}(T_\B) \cup V_{\text{even}}(T_\R) \cup V_1(T)| - |D| \tag{Lemma~\ref{lem. vert. partitioning of T}}\\
        &= |V_{\text{even}}(T_\B) \cup V_{\text{even}}(T_\R) \cup V_1(T)| - |D_\B \cup D_\R \cup V_1(T)|\\
        &= |V_{\text{even}}(T_\B)| + |V_{\text{even}}(T_\R)| - |D_\B| - |D_\R| \\
        &= \dim\left(\frac{R_\B}{\NN_{\text{odd}}(T_\B)}\right) + \dim\left(\frac{R_\R}{\NN_{\text{odd}}(T_\R)}\right)\\
        &= \depth\left(\frac{R_\B}{\NN_{\text{odd}}(T_\B)} \right) + \depth\left(\frac{R_\R}{\NN_{\text{odd}}(T_\R)} \right) \tag{Corollary~\ref{cor. CMness and depth of Delta-forests}}\\
        &= |Z_\B| + |Z_\R|\\
        &= |Z|.
    \end{align*}
    
    Thus $Z$ is a maximal $R/\NN(T)$-regular sequence.
    Writing $V_\B := V_{\text{even}}(T_\B)$ and $V_\R := V_{\text{even}}(T_\R)$, we get
    \begin{align*}
        \tp\left(R/\NN(T) \right) \hspace{-2.75cm}\\
        &= \tp\left(\frac{R}{\NN(T) + \gen{Z}} \right)\\
        &= \tp\left(\frac{\kk[V_\B]}{\NN_{\text{odd}}(T_\B)+\gen{Z_\B}} \otimes_\kk \frac{\kk[V_\R]}{\NN_{\text{odd}}(T_\R)+\gen{Z_\R}}\right)\\
        &= \dim_\kk\left(\Ext_{\kk[V_\B \cup V_\R]}^0\left(\kk, \frac{\kk[V_\B]}{\NN_{\text{odd}}(T_\B)} \otimes_\kk \frac{\kk[V_\R]}{\NN_{\text{odd}}(T_\R)}\right) \right)\\
        &= \dim_\kk\left(\Ext_{\kk[V_\B]}^0\left(\kk, \frac{\kk[V_\B]}{\NN_{\text{odd}}(T_\B)+\gen{Z_\B}}\right) \otimes_\kk\ \Ext_{ \kk[V_\R]}^0\left(\kk, \frac{\kk[V_\R]}{\NN_{\text{odd}}(T_\R)+\gen{Z_\R}}\right)\right)\\
        &= \dim_\kk\left(\Ext_{\kk[V_\B]}^0\left(\kk, \frac{\kk[V_\B]}{\NN_{\text{odd}}(T_\B)+\gen{Z_\B}}\right)\right) \cdot \dim_\kk\left(\Ext_{ \kk[V_\R]}^0\left(\kk, \frac{\kk[V_\R]}{\NN_{\text{odd}}(T_\R)+\gen{Z_\R}}\right)\right)\\
        &= \tp\left(\frac{\kk[V_\B]}{\NN_{\text{odd}}(T_\B)+\gen{Z_\B}} \right) \cdot \tp\left(\frac{\kk[V_\R]}{\NN_{\text{odd}}(T_\R)+\gen{Z_\R}} \right)\\
        &= \tp(R/\NN_{\text{odd}}(T_\B)) \cdot \tp(R/\NN_{\text{odd}}(T_\R))\\
        &= m_\B \cdot m_\R \tag{Theorem~\ref{thm. max. reg. seq. and type of Delta-forest}}
    \end{align*}
as desired.
\end{proof}

We end this chapter with some examples applying Theorem~\ref{thm. CM type of WTD trees}.

\begin{example}
    Let $T$ be the tree from Example~\ref{ex. decomp. of a tree}.
    Consider its blue and red interior graphs $T_\B$ and $T_\R$ in Figure~\ref{fig. interiors of fence tree}.
    \begin{figure}[ht]
        \centering
        \includegraphics{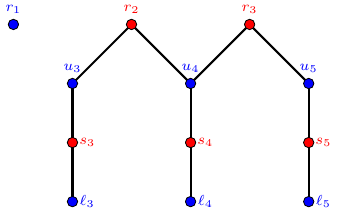}
        \qquad
        \qquad
        \qquad
        \includegraphics{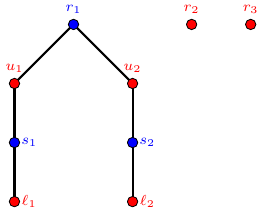}
        \caption{$T_\B$ and $T_\R$ from Example~\ref{ex. decomp. of a tree}}
        \label{fig. interiors of fence tree}
    \end{figure}
    The minimal $V_3$-TDSs of $T_\B$ are $\set{u_3,u_5}$ and $\set{u_4}$.
    The minimal $V_3$-TDSs of $T_\R$ are $\set{u_1}$ and $\set{u_2}$.
    Therefore, we have $\tp(R/\NN(T)) = 2 \cdot 2 = 4$ .
\end{example}

\begin{example}
    Let $T$ be the tree in Example~\ref{ex. T, T_B = T_R} and let $R = \kk[V(T)]$.
    We have 
    $$
    \tp(\kk[V_{\text{even}}(T_\B)]/\NN_{\text{odd}}(T_\B)) = 2
    $$ 
    from Example~\ref{ex. type of T_B, T_B = T_R}.
    Since $T_\B$ and $T_\R$ are isomorphic graphs, we get
    $$
    \kk[V_{\text{even}}(T_\B)]/\NN_{\text{odd}}(T_\B) \cong \kk[V_{\text{even}}(T_\R)]/\NN_{\text{odd}}(T_\R)
    $$
    which gives
    $$
    \tp(\kk[V_{\text{even}}(T_\R)]/\NN_{\text{odd}}(T_\R)) \cong  \tp(\kk[V_{\text{even}}(T_\B)]/\NN_{\text{odd}}(T_\B)) = 2.
    $$
    Thus we get
    \begin{align*}
        \tp(R/\NN(T)) &= \tp(\kk[V_{\text{even}}(T_\B)/\NN_{\text{odd}}(T_\B)]) \cdot \tp(\kk[V_{\text{even}}(T_\R)]/\NN_{\text{odd}}(T_\R))\\ 
        &= 2 \cdot 2 = 4.
    \end{align*}
\end{example}


\printbibliography

\end{document}